\documentclass{amsart}
\usepackage[colorlinks=true,pagebackref,hyperindex,citecolor=blue,linkcolor=blue]{hyperref}
\usepackage[top=1in, bottom=1in, left=1.1in, right=1.1in]{geometry}
\usepackage{amsmath}
\usepackage{amsfonts}
\usepackage{amssymb}
\usepackage{amsthm}
\usepackage{stmaryrd}
\usepackage{braket}
\usepackage[all]{xy}
\usepackage{mathrsfs}
\usepackage{enumitem} 
\usepackage{color}
\usepackage{tikz}
\usepackage{nicefrac, xfrac}
\usepackage[numbers]{natbib}
\usepackage{comment} 
\usepackage[normalem]{ulem}

\usepackage[T1]{fontenc}
\usepackage[utf8]{inputenc}

\newtheorem{theorem}{Theorem}[section]
\newtheorem{corollary}[theorem]{Corollary}

\newtheorem{lemma}[theorem]{Lemma}
\newtheorem{proposition}[theorem]{Proposition}

\newtheorem*{theoremx}{Theorem}

\theoremstyle{definition}

\newtheorem{example}[theorem]{Example}
\newtheorem{remark}[theorem]{Remark}
\numberwithin{equation}{section}

\newcommand{\ZZ}{\mathbb{Z}}
\newcommand{\PP}{\mathbb{P}}

\newcommand{\CC}{\mathbb{C}}

\newcommand{\ctM}{\Theta_{M}}
\newcommand{\mcI}{\mathcal{I}}

\newcommand{\mcO}{\mathcal{O}}

\newcommand{\mcH}{\mathcal{H}}

\newcommand{\mcE}{\mathcal{E}}
\newcommand{\mcF}{\mathscr{F}}

\makeatletter
\def\@tocline#1#2#3#4#5#6#7{\relax
  \ifnum #1>\c@tocdepth
  \else
    \par \addpenalty\@secpenalty\addvspace{#2}%
    \begingroup \hyphenpenalty\@M
    \@ifempty{#4}{%
      \@tempdima\csname r@tocindent\number#1\endcsname\relax
    }{%
      \@tempdima#4\relax
    }%
    \parindent\z@ \leftskip#3\relax \advance\leftskip\@tempdima\relax
    \rightskip\@pnumwidth plus4em \parfillskip-\@pnumwidth
    #5\leavevmode\hskip-\@tempdima
      \ifcase #1
       \or\or \hskip 1.9em \or \hskip 2em \else \hskip 3em \fi%
      #6\nobreak\relax
    \dotfill\hbox to\@pnumwidth{\@tocpagenum{#7}}\par
    \nobreak
    \endgroup
  \fi}
\makeatother

\begin{document}

\title[Foliations on Surfaces]{Effective Results for Foliations on Smooth Projective Complete Intersection Surfaces}

\author[J. Olivares]{Jorge Olivares}
\address{Centro de Investigaci\'on en Matem\'aticas, A.C.
A.P. 402, Guanajuato 36000, México} 
\email{olivares@cimat.mx}

\author[D. Posada-Buriticá]{Daniel Posada-Buriticá{$^1$}}
\address{Centro de Investigaci\'on en Matem\'aticas, A.C.
A.P. 402, Guanajuato 36000, México} 
\email{daniel.posada@cimat.mx}

\thanks{{$^1$}This author was partially supported by CONACYT-SECIHTI Scholarship 842809.}



\begin{abstract}
We study holomorphic foliations on projective spaces that leave smooth 
projective 
complete intersection surfaces $M$ invariant. We determine precisely for which degrees such foliations on $M$ exist. As a consequence, we obtain new bounds for the classical Poincaré problem for smooth projective complete intersection surfaces and prove that previously known bounds for smooth hypersurfaces in $\PP^3$ are optimal. Furthermore, for a foliation $[s]$ on $M$ with isolated singularities and for degrees beyond an explicit bound that we provide, 
we show that a section $s'$ has singular scheme containing that of $s$ if and only if $s'=\phi(s)$ for some global endomorphism $\phi$ of the tangent bundle of $M$.
\end{abstract}

\maketitle

\section*{Introduction}

Holomorphic foliations on compact connected complex manifolds constitute a classical topic in Algebraic Geometry, lying at the intersection of complex geometry, dynamical systems, and the theory of algebraic differential equations. Besides their intrinsic geometric interest, they naturally arise in connection with several classical problems concerning invariant subvarieties, singularities, and vector fields on complex manifolds.

Two fundamental problems
arise naturally when studying holomorphic foliations on a fixed compact connected complex manifold $X$.

\textbf{I)} To decide which line bundles occur as tangent sheaves of holomorphic 
foliations on $X$.

\textbf{II)} Given a holomorphic foliation on $X$ with isolated singularities, to decide the existence of other foliations that share its
singularities with the given one.

The first problem is essentially cohomological, since foliations with tangent sheaf 
$\mathscr{L}$ correspond to nonzero global sections of 
$H^0(X,\Theta_X \otimes\mathscr{L}^\vee)$, modulo multiplication by a nonzero scalar. 
In other words, the space $\mathscr{F}ol(\mathscr{L},X)$ of foliations in $X$ with tangent sheaf $\mathscr{L}$ is precisely $\PP H^0(X,\Theta_X \otimes\mathscr{L}^\vee)$.
The second problem concerns the extent to which the geometry of the singular scheme determines the corresponding foliation. More precisely, if 
$s\in H^0(X,\Theta_X\otimes\mathscr L^\vee)$ is non-zero
and has isolated singularities, we denote by $Z_s$ its singular 
scheme and say that the foliation defined by $s$ is \emph{uniquely determined by its singular scheme} if every section $s'$ whose singular scheme  $ Z_{s'} $ is equal to
$Z_s$ is, necessarily, a scalar multiple of $s$. A detailed discussion of these notions
is recalled in Section \ref{Bckgnd}.

For complex projective spaces, 
both problems are solved:
Indeed, the Euler sequence shows that all
foliations on $\PP^n$ have tangent sheaf equal to
$\mathcal O_{\PP^n}(1-d)$, 
for $d\ge0$ (see Subsection~\ref{Subsec:FolinPn}), 
and moreover,
a foliation on $\PP^n$ with isolated singularities is uniquely determined by
its singular scheme, whenever $d\ge2$ (see \cite{CampilloOlivares}). This result
was previously proved in \cite{GMKempf} for reduced singularities and later in
\cite{Polarity} for the case of $\PP^2$.
Both problems have been also solved for foliations on Hirzebruch surfaces in \cite{Hirzebruch}.

The second problem has also been investigated in several other settings: Sufficient conditions ensuring that a foliation or a higher-rank distribution is determined by its singular scheme have been established for codimension one foliations and higher-rank distributions on projective spaces (see \cite{G-P,C-F-N-V,A-C}).

Both problems were 
solved in the special case where $X$ is a projective complete intersection K3 surface 
in \cite{OlivaresPosada}. The results obtained in that setting are recovered as 
particular cases of the more general framework developed in the present work.

The first problem is closely related to another classical question in the theory of holomorphic foliations, namely the \emph{Poincaré problem}. In its modern formulation, this problem seeks
for bounds on the degree of smooth projective subvarieties invariant under a foliation of prescribed degree on projective space. Therefore, any necessary condition for the existence of invariant subvarieties immediately yields a solution to the Poincaré problem for the corresponding class of subvarieties. Explicit bounds are currently known in the cases of smooth projective hypersurfaces and smooth projective complete intersections of odd dimension, see \cite{MarcioInv,MarcioAnn}.

Motivated by these problems, in this paper we study both problems for holomorphic foliations on smooth projective complete intersection surfaces. Let $X=V(f_1,\ldots,f_c)\overset{i}{\hookrightarrow}\PP^n$ be a smooth complete intersection surface of type $(d_1,\ldots,d_c)$. Throughout the paper we will consider foliations on $ X $ with tangent sheaf 
$ \mathscr{ L }_d = i^{\ast}\mathcal{O}_{\PP^n}(1-d) $, 
where the dimension $ n $ depends on the complete intersection type of $ X $.
This means that we shall consider foliations 
$ \mathscr{F} = [ s ]$ on $ \PP^n $, with $ s \in H^0(\PP^n,\Theta_{\PP^n}\otimes\mathcal{O}_{\PP^n}(d-1))$,
such that its restriction $ i^{\ast} s $ to $ X $ factors through the tangent sheaf
$ \Theta_X $ of $ X $:
$ i^{\ast} s \in H^0( X, \Theta_X \otimes \mathscr{L}_d^{\vee}) $.
That is, foliations $ \mathscr{F} $ which leave $ X $ invariant.
Geometrically, these correspond to polynomial vector fields on $\PP^n$ tangent to $X$ at every point.

Our main results are summarized in the following statement:
\begin{theoremx}
Let $ X=V(f_1, \ldots, f_c) \overset{ i }{\hookrightarrow} \PP^n $ be a smooth complete
intersection surface of type $(d_1, \ldots, d_c)$ in $\PP^n$, with 
$2 \leq d_1 \leq \cdots \leq d_c$, and consider the invertible sheaves 
$\mathscr{L}_d=i^{\ast}\mathcal{O}_{\PP^n}(1-d)$ on $X$, with 
$ d \geq 0 $. Then,
\begin{itemize}
\item[ \textbf{A} ] \label{Theorem A}
Foliations on $ X $ with tangent sheaf $\mathscr{L}_d $ exist only for $ d \geq \sum_{i=1}^c d_i -n+2$.
\item[ \textbf{B} ]  \label{Theorem B}
Let $d > \sum_{d=1}^c d_i -n+d_c$ and
let $ s \in H^0(X,\Theta_X \otimes \mathscr{L}_d^{\vee}) $ be a nonzero
section with isolated singularities. Then, a nonzero section 
$ s' \in H^0(X,\Theta_X \otimes \mathscr{L}_d^{\vee})$ satisfies that 
$ Z_{s'} \supset Z_s $ if and only if $s'=\phi(s)$, for some global endomorphism 
$\phi$ of $\Theta_X$. In particular, if $\Theta_X$ is \textit{simple} (which by 
definition means that every global endomorphism of $\Theta_X$ is an 
scalar multiple of the identity) then there is a unique foliation $ [s] $ having 
$ Z_s $ as singular scheme.

\end{itemize}
\end{theoremx}

Theorem \textbf{A} provides a complete solution to the existence problem for smooth projective complete intersection surfaces $X$. In particular, it provides explicit necessary and sufficient conditions for the existence of holomorphic foliations on $X$ with tangent sheaf $\mathscr{L}_d$. As an immediate consequence, it yields new solutions to the Poincaré problem for smooth projective complete intersection surfaces (Corollary \ref{PPSforCISurfaces}), extending the results of \cite{MarcioAnn}, where only complete intersections of odd dimension 
were considered. Moreover, for smooth 
hypersurfaces in $\PP^3$, Corollary \ref{PPSforCISurfaces} shows that the bound 
previously obtained in \cite{MarcioInv} is not only sufficient but it is also necessary.

Theorem \textbf{A} becomes particularly significant for \emph{very general} complete intersection surfaces $ X $:
by \cite[Theorem 0.1]{VeryGeneral}, every 
such an $ X $ has Picard rank equal to one 
(except for quadric and cubic surfaces in $\PP^3$, and the complete intersection of two quadrics in $\PP^4$).  In consequence, every line bundle on $X$ is of the form $i^*\mathcal{O}_{\PP^n}(k)$, for some $k\in\mathbb Z$. Since the space of foliations by 
curves on $X$ is given by
$$
\bigsqcup_{\mathscr L\in\mathrm{Pic}(X)}
\PP H^0(X,\Theta_X\otimes\mathscr L^\vee)
=
\bigsqcup_{d\in\mathbb Z}
\PP H^0(X,\Theta_X\otimes i^*\mathcal O_{\PP^n}(d-1)),
$$
by \cite[Theorem 9]{Universal},
Theorem~\textbf{A} gives \textit{all} tangent sheaves of holomorphic foliations on 
every such an $ X $, 
with the only exceptions mentioned above.



Let $ s \in H^0(X, \Theta_X \otimes i^{\ast} \mathcal{O}_{\PP^n}(d-1)) $ be a nonzero section
with isolated singularities. The \textit{degree} $ \mathrm{deg}(\mathrm{Z_s}) $
of the singular scheme of $ s $ is computed in Theorem \ref{eq:NumberofSing} below.
Let $\Lambda:=\mathrm{deg}(\mathrm{Z_s})$.


The assignment $[s]\mapsto Z_s$ defines a map from the space 
$ \mathscr{F}ol^{\mathrm{iso}}(i^{\ast} \mathcal{O}_{\PP^n}(1-d),X) $
of foliations on $X$ with isolated singularities and
tangent sheaf $i^{\ast} \mathcal{O}_{\PP^n}(1-d)$, 
to the Hilbert scheme 
$ \mathrm{Hilb}^{\Lambda}(X) $
of $\Lambda$ points in $X$:
$$
    \begin{array}{r@{\hspace{5pt}}c@{\hspace{5pt}}c@{\hspace{5pt}}l}
    \Gamma:& \mathscr{F}ol^{\mathrm{iso}}(i^{\ast} \mathcal{O}_{\PP^n}(1-d),X) &\longrightarrow& \mathrm{Hilb}^{\Lambda}(X)\\
    &\mathscr{F}=[s] &\longmapsto& Z_s.\\
    \end{array}
$$

Let $d>\sum_{i=1}^c d_i-n+d_c$. In these cases, Theorem \textbf{B} is related with
the fibers of $\Gamma$, namely, with those foliations $ [s'] $ such that $Z_{s'}=Z_s$:

If $\Theta_X$ is simple, Theorem \textbf{B} shows that every fiber of 
$\Gamma$ consists of a single point (see Remark \ref{Remark:CharSections} below), 
which means that $\Gamma$ is injective: this allows one to study foliations $ [s] $
on $ X $ through the study of its singular schemes $ Z_s $ (for instance, 
in a similar way as in \cite{Claudia,ClaudiaRubi,ClaudiaRonzon}, where the goal is
to classify foliations on $\PP^2$ through the geometry of its singular schemes).

If $\Theta_X$ is not simple, Theorem \textbf{B} and Theorem 
\ref{Thm:ESSifInvertible} 
give a partial description of the fibers of $\Gamma$, namely: if $s'=\phi(s)$
for a global \textit{automorphism} (invertible endomorphism)
$\phi \neq \lambda \cdot \mathrm{id}_{\Theta_X}$
of $ \Theta_X $, then the foliations $ [s] $ and $ [s'] $ are different and 
have the same singular scheme $ Z_{s'}=Z_s $.




The paper is organized as follows: Section \ref{Bckgnd} reviews the background material on holomorphic foliations and complete intersections needed throughout the paper. In particular, we collect the cohomological tools required for the proofs of the main theorems. In Section~\ref{Sec:Existence} we determine the dimension of $H^0(X,\Theta_X\otimes i^*\mathcal O_{\PP^n}(d-1))$, thereby proving Theorem \textbf{A}. In Section~\ref{Sec:Degree} we compute the number of singularities (counting multiplicities) of the foliations under consideration. Section \ref{Sec:SingScheme} contains the proof of Theorem \textbf{B}. Finally, the study of global endomorphisms of the tangent bundle 
$\Theta_X$ of $X$ is carried out in Section \ref{Sec:GlobalEnd}.

\section{Preliminaries}\label{Bckgnd}
 Let $ \PP^n = \textup{Proj} \ \CC[ x_0, \dots, x_n ] $ be the complex projective space of dimension 
 $ n \geq 2 $, and let $ \mcO_{\PP^n}$, $ \Theta_{\PP^n} $, $ \Omega^1_{\PP^n} $ and $ \mathcal{H} $
 denote its structure, tangent, cotangent and hyperplane sheaves, respectively. For an 
 $ \mcO_{\PP^n}$-sheaf $ \mathcal{E} $, we will write 
 $ \mcE( d ) $ for $ \mcE \otimes \mcH^{\otimes d}$, if $ d \geq 0 $, and 
$ \mcE \otimes (\mcH^{\vee})^{\otimes |d|}$, if $ d < 0 $. Similar notation will be adopted
for a submanifold (or subvariety) $ X \subset \PP^n $. Any vector bundle
will be identified with its locally free sheaf of sections. Finally, throughout the paper we use the classical combinatorial convention for binomial coefficients, namely $\binom{n}{d}=\frac{n!}{d!(n-d)!}$ if $0 \leq d \leq n$, and $0$ otherwise.

\subsection{\texorpdfstring{Foliations by Curves}{Foliations by Curves}:}\label{Subsec:FolinPn}
Let $ M $ be a compact connected complex manifold of dimension $ n $. Recall that a
holomorphic foliation by curves $ \mathscr{F} $ on $ M $  (simply \emph{a foliation on} $ M $
in the sequel) may be defined by
non-identically zero holomorphic vector fields $ X_i $ defined on a covering $ \{ V_i \} $ of $ M $
such that in each overlapping set $ V_i \cap V_j \neq \emptyset $ we have
\begin{equation}\label{cocicle}
   X_i = \xi_{ij} X_j,
\end{equation}
where $ \xi_{ij} $ is a never vanishing holomorphic function. If
$ L^{\vee} $ denotes the holomorphic line bundle constructed with the cocycle $( \xi_{ij} )$,
and $\mathscr{L}^{\vee}$ its corresponding invertible sheaf,
then the $ X_i $'s give rise to a global section
$ s \in H^{0}(  M , \ctM \otimes \mathscr{L}^{\vee}) $ or to a global section in
$ H^{0}( M, \mathscr{H}om_{ \mcO_M} (\mathscr{L}, \ctM) )$, where
$ \ctM $ is the tangent sheaf of $ M $ and $\mathscr{L}$ is the dual
of $\mathscr{L}^{\vee}$.
 Two global sections (in the corresponding spaces) define the same
foliation if and only if one is a non-zero scalar multiple of the other.

The line bundle
$ L^{\vee} $ just defined 
will be called the \emph{cotangent bundle of} $\mathscr{F}$ and its dual 
$ L $, \emph{its tangent bundle}.
Hence, the space $\mathscr{F}ol(\mathscr{L}, M) $ of foliations
$ \mcF $ with tangent bundle $ L $ (or tangent sheaf $ \mathscr{L} $)
is $ \PP H^{0}( M, \mathscr{H}om(\mathscr{L} , \ctM) )$.
Such an $ \mcF $ corresponds to a foliation with cotangent bundle $ L^{\vee}$ (or cotangent sheaf 
$ \mathscr{L}^{\vee} $)
by regarding it as the class $ [ s ] \in \PP H^{0}( M, \ctM \otimes \mathscr{L}^{\vee}) $ 
of a global section $ s \in H^{0}( M, \ctM \otimes \mathscr{L}^{\vee}) $.

Given a global section
$ s \in H^{0}( M, \mathscr{H}om_{ \mcO_M} (\mathscr{L} , \ctM) )$,
the scheme $ Z = Z_s $ of those points 
$ p\in M $ where the induced morphism
$\mathscr{L}_p\rightarrow \Theta_{M,p}$ is zero will be referred to as
the \emph{singular scheme} of $ s $: its ideal sheaf 
$ \mathcal{I}_Z \subset \mcO_M $
is the sheaf obtained by gluing the ideals $ (a_i^1, \ldots, a_i^n) \subset \mcO( V_i ) $,
where $ a_i^1,\ldots, a_i^n $ are the components of the vector field $ X_i $ that defines
$ s $ on the open set $ V_i $, as described nearby \eqref{cocicle}. It is clear that the singular scheme
of $ \mcF = [s] $ is the singular scheme of any section in $ [s] $.

We say that $ \mcF = [s]$ has \emph{isolated singularities} if
$ \text{dim } Z = 0 $. We will say that a foliation $\mathscr{F}=[s]$ is uniquely determined by its singular scheme $Z_s$, if it is the unique one with that singular scheme. More precisely, if there exist another section $s'$ with isolated singularities and satisfying that $Z_{s'} \supset Z_s$, then $s$ and $s'$ define the same foliation. 

Foliations by curves 
with tangent sheaf 
$ \mathcal{O}_{\PP^n}(1-d) $ on $\PP^n$ arise
from polynomial homogeneous vector fields of degree $ d $  in $ \CC^{n+1} $ through the following 
construction: Let $ \pi : \CC^{n+1} \setminus \{0\} \longrightarrow \PP^n ; \; p\mapsto [p] $, be the
projection map. Tensor the Euler sequence in $\PP^n$ \cite[p. 3]{Okonek} with the invertible sheaf
$\mathcal{O}_{\PP^n}(d)$ to obtain the exact sequence
\begin{equation*}
    0 \longrightarrow \mathcal{O}_{\PP^n}(d-1) \longrightarrow \mathcal{O}_{\PP^n}(d)^{\oplus (n+1)} \overset{ \pi_{*} }{\longrightarrow} \Theta_{\PP^n}(d-1) \longrightarrow 0,
\end{equation*}
whose long exact cohomology sequence is:
\begin{equation}\label{twisteulseq}
    0 \to H^0(\PP^n,\mathcal{O}_{\PP^n}(d-1)) \to H^0 (\PP^n,\mathcal{O}_{\PP^n}(d))^{\oplus (n+1)} 
    \overset{ \pi_{*} }{\longrightarrow} H^0(\PP^n,\Theta_{\PP^n}(d-1)) \to 0,
\end{equation}
since $H^1(\PP^n, \mathcal{O}_{\PP^n}(d-1))=0$ (see Lemma \ref{LemmBottFla} below).

It follows from \eqref{twisteulseq} that any section
$ s \in H^0( \PP^n, \Theta_{\PP^n}( d-1 ) ) \cong 
H^0( \PP^n, \mathscr{H}om( \mcO_{\PP^n}( 1-d ), \Theta_{\PP^n} ) ) $
 (and hence,
 any foliation on $\PP^n$ with tangent sheaf $ \mathcal{O}_{\PP^n}(1-d) $) 
 is defined by a polynomial homogeneous vector field of degree $d$ in $\CC^{n+1}$, where 
 two such vector fields $F$ and $G$ define the same section $ s $ iff $G= F + h \cdot R$, where
 $R=\sum_{i=0}^n x_i \frac{\partial}{\partial x_i}$ is the radial vector field in $\CC^{n+1}$ and
  $h$  is  some homogeneous polynomial of degree $d-1$. In particular, no such global sections
  $ s $ exist for $ d < 0 $.

\subsection{Smooth projective complete intersections:}
Let $X=V(f_1, \ldots, f_c) \subset \PP^n$ be a smooth complete intersection of codimension $ c $, where each $f_i$ is a homogeneous polynomial of degree $d_i$. We shall refer to such an $X$ as a 
\emph{smooth complete intersection of type $(d_1, \ldots, d_c)$ in $ \PP^n $}. We assume that 
$ d_i > 1 $ for every $i$ to guarantee that $ n $ is minimal.

The augmented Koszul complex \cite[Chapter IV, $\mathsection 2$]{RRAlgebra} 
associated to such an $ X $
gives the following resolution of the \emph{ideal sheaf} $\mathcal{I}_X$ of $X$:
\begin{equation}\label{KoszCxSh}
0 \longrightarrow \mathcal{O}_{\PP^n}\left(-\sum_{i=1}^c d_i \right) \overset{\varphi_c}{\longrightarrow} \cdots \overset{\varphi_3}{\longrightarrow} \bigoplus_{1 \leq j < k \leq c} \mathcal{O}_{\PP^n}(-d_j-d_k) \overset{\varphi_2}{\longrightarrow}\bigoplus_{j=1}^c \mathcal{O}_{\PP^n}(-d_j) \overset{\varphi_1}{\longrightarrow}\mathcal{I}_X \longrightarrow 0.
\end{equation}
The long exact sequence \eqref{KoszCxSh} breaks into the short exact sequences given by 
\begin{equation}\label{eq:SystemIdeal}
    \begin{split}
    0& \longrightarrow \mathrm{Im}\ \varphi_2 \longrightarrow \bigoplus_{j=1}^c \mathcal{O}_{\PP^n}(-d_j) \longrightarrow \mathcal{I}_X \longrightarrow 0\\
    0& \longrightarrow \mathrm{Im}\ \varphi_3 \longrightarrow \bigoplus_{\substack{ J \subset \{1, \ldots, c\} \\ |J|=2}} \mathcal{O}_{\PP^n} \left(- \sum_{i \in J} d_i \right) \longrightarrow \mathrm{Im}\ \varphi_2 \longrightarrow 0\\
    0& \longrightarrow \mathrm{Im}\ \varphi_4 \longrightarrow \bigoplus_{\substack{ J \subset \{1, \ldots, c\} \\ |J|=3}} \mathcal{O}_{\PP^n} \left(- \sum_{i \in J} d_i \right) \longrightarrow \mathrm{Im}\ \varphi_3 \longrightarrow 0\\
    \vdots&\\
    0& \longrightarrow \mathcal{O}_{\PP^n}(-\lambda) \longrightarrow \bigoplus_{\substack{ J \subset \{1, \ldots, c\} \\ |J|=c-1}} \mathcal{O}_{\PP^n} \left(- \sum_{i \in J} d_i \right) \longrightarrow \mathrm{Im}\ \varphi_{c-1} \longrightarrow 0,
\end{split}
\end{equation}
where $\lambda= \sum_{j=1}^c d_j$.
It is not hard to see from \eqref{KoszCxSh} that
the \emph{normal sheaf} $\mathcal{N}_{X/\PP^n}$ of $X$ in $\PP^n$, defined as the dual of the conormal sheaf $\mathcal{I}_X/ \mathcal{I}_X^2$, is given by
\begin{equation}\label{NormalBundle}
\mathcal{N}_{X/\PP^n} \cong \bigoplus_{i=1}^ci^{\ast} \mathcal{O}_{\PP^n}(d_i).
\end{equation}
This sheaf may be incorporated in a sequence that involves the tangent sheaf of $X$ with restriction to $X$ of the tangent sheaf of the ambient space, given by
\begin{equation}\label{eq:NormalSequence}
    0 \longrightarrow \Theta_X \longrightarrow i^{\ast} \Theta_{\PP^n} \longrightarrow \bigoplus_{i=1}^ci^{\ast} \mathcal{O}_{\PP^n}(d_i) \longrightarrow 0,
\end{equation}
known as the \emph{normal sequence} of $X$ in $\PP^n$.

By \cite[Ex. 8.4 (e)]{Hartshorne}, the canonical sheaf of $X$ is given by
\begin{equation}\label{CanonicalBundle}
    \omega_X \cong i^{\ast}\mathcal{O}_{\PP^n}\left(-n-1+\sum_{i=1}^c d_i\right),
\end{equation}
which will play a key role in relating vector fields and differential forms on $X$, through the formula
\begin{equation}\label{eq:FormTanCot}
    \Omega^{\mathrm{dim}(X)-1}_X \cong \Theta_X \otimes i^{\ast}\mathcal{O}_{\PP^n}\left(-n-1+\sum_{i=1}^c d_i\right),
\end{equation}
given by the isomorphism $\Omega^{\mathrm{dim}(X)-1}_X \cong \Theta_X \otimes \omega_X$ in \cite[Theorem 3.6.1]{Hirzebruch}.

\medskip
\subsection{Some cohomological formulas:}
The following formulas provide explicit expressions for the cohomology groups that will be used throughout the proofs of the main theorems.

We begin by recalling the following classical result (see \cite[p.~4]{Okonek}): 
\begin{lemma}[Bott formulas]\label{LemmBottFla}
$$h^q(\PP^n, \Omega^p_{\PP^n}(k))= 
    \begin{cases}
    \binom{k+n-p}{n-p}\binom{k-1}{p}, & \text{for } q=0, 0 \leq p \leq n, k>p,\\
    1, & \text{for } k=0, 0 \leq p=q \leq n,\\
    \binom{-k+p}{p}\binom{-k-1}{n-p}, & \text{for } q=n, 0 \leq p \leq n, k<p-n,\\
    0, & \text{otherwise.}
    \end{cases}$$
In particular, for $p=0$ we have
$$h^q(\PP^n, \mathcal{O}_{\PP^n}(k))= 
    \begin{cases}
    \binom{k+n}{n}, & \text{for } q=0, k \geq 0,\\
    \binom{-k-1}{n}, & \text{for } q=n, k \leq -n-1,\\
    0, & \text{otherwise.}
    \end{cases}$$
\end{lemma}  
We follow by recalling two standard facts relating the cohomology of a sheaf on a closed subvariety with the cohomology of its extension by zero on the ambient space.
\begin{proposition}\label{PropHart}
Let $X$ be a closed subset of $Y$, let $\mcE$ and $ \mathcal{F} $ be locally free sheaves on $X$ 
and $Y$ respectively,  and let $i : X \to Y$ be the inclusion. Then:
\begin{itemize}
    \item[(1) ] $H^k(X,\mcE)=H^k(Y,i_{\ast}\mcE)$, where $i_{\ast}\mcE$ is the extension of 
    $\mcE$ by zero outside $X$. In particular, 
    $ H^k(X,i^{\ast}\mathcal{F}) = H^k(Y,i_{\ast}i^{\ast}\mathcal{F}) $.
    \item[(2) ] The sequence 
    $$ 0 \longrightarrow \mathcal{F} \otimes_{\mathcal{O}_Y} \mathcal{I}_X 
     \longrightarrow \mathcal{F} \longrightarrow i_{\ast}i^{\ast} \mathcal{F} \longrightarrow 0 $$
    is exact.
\end{itemize}
\begin{proof}
For the proof of \textit{(1)}, see \cite[III Lemma 2.10]{Hartshorne}.

Now we prove \textit{(2)}:
Consider the exact sequence on $Y$ of the ideal sheaf of $X$ \cite[p. 115]{Hartshorne}
$$
0 \longrightarrow \mathcal{I}_X \longrightarrow \mathcal{O}_Y \longrightarrow i_{\ast}\mathcal{O}_X \longrightarrow 0,
$$
and twist it by the sheaf $\mathcal{F} $ to obtain 
$$
0 \longrightarrow \mathcal{F} \otimes_{\mathcal{O}_Y} \mathcal{I}_X \longrightarrow \mathcal{F} 
    \longrightarrow \mathcal{F} \otimes_{\mathcal{O}_Y} i_{\ast}\mathcal{O}_X  \longrightarrow 0.
$$
It follows from the Projection Formula \cite[Ex. 5.1 (d)]{Hartshorne} that
$ \mathcal{F} \otimes_{\mathcal{O}_Y} i_{\ast}\mathcal{O}_X \cong i_{\ast}i^{\ast} \mathcal{F} $.
\end{proof}
\end{proposition}

For the following result, see \cite[\S 2.2]{Cr-Es}:
\begin{lemma}\label{aCM}
Let $ X \subset \PP^n $ be a complete intersection with ideal sheaf $\mathcal{I}_X$. Then
$ H^j( \PP^n, \mathcal{I}_X( m ) ) = 0 $ for every $ m \in \ZZ $ and 
$ j = 1, \dots, \textup{dim}(X) $.
\end{lemma}

We begin by computing the cohomology of the restriction of the structure sheaf of 
$\mathbb P^n$ to $X$.

\begin{proposition}\label{Prop:RestrictedBottFirst}
Let $X=V(f_1, \ldots, f_c) \subset \PP^n$, with $n \geq 3$, be a smooth complete intersection surface of type $(d_1, \ldots, d_c)$ in $\PP^n$. Then, for every integer $k$,
    \begin{align*}
        &h^0(X,i^{\ast}\mathcal{O}_{\PP^n}(k))=
        \sum_{J\subset \{1,\ldots,c\}} (-1)^{|J|} \ \binom{-\sum_{i\in J}d_i+k+n}{n}, \textup{ and}\\
        &h^1(X,i^{\ast}\mathcal{O}_{\PP^n}(k))=0.
    \end{align*}
\begin{proof}
If $n=3$, Proposition \ref{PropHart} yields the exact sequence
$$0 \to \mathcal{O}_{\PP^3}(k-d_1) \to \mathcal{O}_{\PP^3}(k) \to i_{\ast} i^{\ast}\mathcal{O}_{\PP^3}(k) \to 0.$$
Applying Lemma \ref{LemmBottFla} to the associated sequence in cohomology, we obtain
\begin{align*}
    h^0(i_{\ast} i^{\ast}\mathcal{O}_{\PP^3}(k))=&h^0(\mathcal{O}_{\PP^3}(k))-h^0(\mathcal{O}_{\PP^3}(k-d_1))\\
    =&\binom{k+3}{3}-\binom{-d_1+k+3}{3},
\end{align*}
which agrees with the stated formula.

Assume now that $n \geq 4$, by Proposition \ref{PropHart}
$$0 \to \mathcal{I}_X(k) \to \mathcal{O}_{\PP^n}(k) \to i_{\ast} i^{\ast}\mathcal{O}_{\PP^n}(k) \to 0.$$
Since $h^{q}(\PP^n, \mathcal{I}_X(k))=0$ for $q=1,2$ by Lemma \ref{aCM}, the associated sequence in cohomology yields
\begin{equation}\label{eq:GoalsPZero}
    \begin{split}
        &h^{0}( i_{\ast} i^{\ast}\mathcal{O}_{\PP^n}(k))= h^0(\mathcal{O}_{\PP^n}(k))- h^0(\mathcal{I}_X(k)),\\
        &h^{1}( i_{\ast} i^{\ast}\mathcal{O}_{\PP^n}(k)) =0.
\end{split}
\end{equation}
To compute $h^0(\mathcal{I}_X(k)$, tensor the short exact sequences in \eqref{eq:SystemIdeal} with $\mathcal{O}_{\PP^n}(k)$ to obtain 
$$0 \to \mathrm{Im}\ \varphi_{s+1} \otimes \mathcal{O}_{\PP^n}(k) \to \bigoplus_{\substack{ J \subset \{1, \ldots, c\} \\ |J|=s}} \mathcal{O}_{\PP^n} \left(- \sum_{i \in J} d_i +k \right) \to \mathrm{Im}\ \varphi_s \otimes \mathcal{O}_{\PP^n}(k) \to 0,$$
for every $s=1, \ldots c-1$.
By Lemma \ref{LemmBottFla} we have that 
\begin{equation*}
        \bigoplus_{J \subset \{1, \ldots, c\}}H^q\left(\mathcal{O}_{\PP^n}\left(-\sum_{i \in J}d_i+k\right)\right)=0, \quad \textup{ for all } k \textup{ and for } q=1,2.
    \end{equation*}
Therefore, we have the isomorphisms
$$H^1(\mathrm{Im} \ \varphi_k \otimes \mathcal{O}_{\PP^n}(k)) \cong H^2( \mathrm{Im} \ \varphi_{k+1} \otimes \mathcal{O}_{\PP^n}(k)), \textup{ for } k=1, \ldots, c-1,$$
where $H^1( \mathrm{Im} \ \varphi_1 \otimes \mathcal{O}_{\PP^n}(k)) \cong H^1( \mathcal{I}_X(k))$ and $H^2( \mathrm{Im} \ \varphi_{c} \otimes \mathcal{O}_{\PP^n}(k)) \cong H^2( \mathcal{O}_{\PP^n}(-\lambda +k))\cong 0$. This concludes that $H^1( \mathrm{Im} \ \varphi_k \otimes \mathcal{O}_{\PP^n}(k))=0$ for $k=1, \ldots, c$ and we obtain the associated cohomology sequence in $\PP^n$
$$0 \to H^0(\mathrm{Im}\ \varphi_{s+1} \otimes \mathcal{O}_{\PP^n}(k)) \to \bigoplus_{\substack{ J \subset \{1, \ldots, c\} \\ |J|=s}} H^0\left(\mathcal{O}_{\PP^n} \left(- \sum_{i \in J} d_i +k \right) \right) \to H^0(\mathrm{Im}\ \varphi_s \otimes \mathcal{O}_{\PP^n}(k)) \to 0,$$
for every $s=1, \ldots, c$.
Consequently, 
$$h^0(\mathrm{Im} \ \varphi_{s} \otimes \mathcal{O}_{\PP^n}(k))=\sum_{\substack{ J \subset \{1, \ldots, c\} \\ |J|=s}} h^0\left(\mathcal{O}_{\PP^n} \left(- \sum_{i \in J} d_i +k\right)\right)-h^0(\mathrm{Im} \ \varphi_{s+1} \otimes \mathcal{O}_{\PP^n}(k)),$$
for every $s=1, \ldots, c-1$. Iterating this identity gives 
$$h^{0}( \mathcal{I}_X(k))=\sum_{s=1}^c (-1)^{s+1} \ \sum_{\substack{ J \subset \{1, \ldots, c\} \\ |J|=s}} h^0\left(\mathcal{O}_{\PP^n} \left(- \sum_{i \in J} d_i +k\right)\right),$$
Finally, combining this expression with \eqref{eq:GoalsPZero} and applying Lemma \ref{LemmBottFla}, we obtain
\begin{align*}
    h^{0}( i_{\ast} i^{\ast}\mathcal{O}_{\PP^n}(k))
    &= h^0(\mathcal{O}_{\PP^n}(k))- h^0(\mathcal{I}_X(k))\\
    &=\sum_{J\subset \{1,\ldots,c\}} (-1)^{|J|} \binom{-\sum_{i\in J}d_i+k+n}{n}.
\end{align*}
as required.
\end{proof}
\end{proposition}

\begin{proposition}\label{Prop:h0Minush1p1}
    Let $X=V(f_1, \ldots, f_c) \subset \PP^n$, with  $n \geq 4$, be a smooth complete intersection surface of type $(d_1, \ldots, d_c)$ in $\PP^n$. Then, for every integer $k$,
    \begin{align*}
    h^0(\Omega^1_{\PP^n}(k)& \otimes \mathcal{I}_X)-h^1(\Omega^1_{\PP^n}(k) \otimes \mathcal{I}_X)\\
    &=\sum_{s=1}^c (-1)^{s+1} \ \sum_{\substack{ J \subset \{1, \ldots, c\} \\ |J|=s}} \left[ h^0 \left(\Omega^1_{\PP^n} \left(- \sum_{i \in J} d_i +k\right)\right)- h^1 \left(\Omega^1_{\PP^n} \left(- \sum_{i \in J} d_i +k\right)\right) \right].
\end{align*}
\begin{proof}
Tensor the short exact sequences in \eqref{eq:SystemIdeal} with $\Omega^1_{\PP^n}(k)$ to obtain 
$$0 \to \mathrm{Im}\ \varphi_{s+1} \otimes \Omega^1_{\PP^n}(k) \to \bigoplus_{\substack{ J \subset \{1, \ldots, c\} \\ |J|=s}} \Omega^1_{\PP^n} \left(- \sum_{i \in J} d_i +k \right) \to \mathrm{Im}\ \varphi_s \otimes \Omega^1_{\PP^n}(k) \to 0,$$
for every $s=1, \ldots, c-1$.
By Lemma \ref{LemmBottFla} we have that 
\begin{equation*}
        \bigoplus_{J \subset \{1, \ldots, c\}}H^q\left(\Omega^1_{\PP^n}\left(-\sum_{i \in J}d_i+k\right)\right)=0, \quad \textup{ for all } k \textup{ and for } q=2,3.
    \end{equation*}
Therefore, we have the isomorphisms
$$H^2(\mathrm{Im} \ \varphi_k \otimes \Omega^1_{\PP^n}(k)) \cong H^3( \mathrm{Im} \ \varphi_{k+1} \otimes \Omega^1_{\PP^n}(k)), \textup{ for } k=1, \ldots, c-1,$$
where $H^2( \mathrm{Im} \ \varphi_1 \otimes \Omega^1_{\PP^n}(k)) \cong H^2( \mathcal{I}_X \otimes \Omega^1_{\PP^n}(k))$ and $H^3( \mathrm{Im} \ \varphi_{c} \otimes \Omega^1_{\PP^n}(k)) \cong H^3( \Omega^1_{\PP^n}(-\lambda +k))\cong 0$. This concludes that $H^2( \mathrm{Im} \ \varphi_k \otimes \Omega^1_{\PP^n}(k))=0$ for $k=1, \ldots, c$, and we obtain the associated cohomology sequence in $\PP^n$
\begin{align*}
    0 \to H^{0}(\mathrm{Im} \ \varphi_{s+1} \otimes \Omega^1_{\PP^n}(k)) \to \bigoplus_{\substack{ J \subset \{1, \ldots, c\} \\ |J|=s}} H^0 \left(\Omega^1_{\PP^n} \left(- \sum_{i \in J} d_i +k\right)\right) \to H^0(\mathrm{Im} \ \varphi_{s} \otimes \Omega^1_{\PP^n}(k)) \to&\\
    \to H^{1}(\mathrm{Im} \ \varphi_{s+1} \otimes \Omega^1_{\PP^n}(k)) \to \bigoplus_{\substack{ J \subset \{1, \ldots, c\} \\ |J|=s}} H^1 \left(\Omega^1_{\PP^n} \left(- \sum_{i \in J} d_i +k\right)\right) \to H^1(\mathrm{Im} \ \varphi_{s} \otimes \Omega^1_{\PP^n}(k)) \to 0,&
\end{align*}
for every $s=1, \ldots, c-1$. Taking the alternating sum of dimensions in the corresponding long exact sequences, we obtain
\begin{equation}\label{eq:p1q1formula}
\begin{split}
    h^0(\mathrm{Im} \ \varphi_{s} &\otimes \Omega^1_{\PP^n}(k))-h^1(\mathrm{Im} \ \varphi_{s} \otimes \Omega^1_{\PP^n}(k))\\
    =&h^{1}(\mathrm{Im} \ \varphi_{s+1} \otimes \Omega^1_{\PP^n}(k))-h^{0}(\mathrm{Im} \ \varphi_{s+1} \otimes \Omega^1_{\PP^n}(k))\\
    &+\sum_{\substack{ J \subset \{1, \ldots, c\} \\ |J|=s}} \left[ h^0 \left(\Omega^1_{\PP^n} \left(- \sum_{i \in J} d_i +k\right)\right)- h^1 \left(\Omega^1_{\PP^n} \left(- \sum_{i \in J} d_i +k\right)\right) \right],
\end{split}
\end{equation}
for every $s =1, \ldots, c-1$. Applying \eqref{eq:p1q1formula} iteratively for every $s$ we conclude that
\begin{align*}
    h^0(\Omega^1_{\PP^n}(k)& \otimes \mathcal{I}_X)-h^1(\Omega^1_{\PP^n}(k) \otimes \mathcal{I}_X)\\
    &=\sum_{s=1}^c (-1)^{s+1} \ \sum_{\substack{ J \subset \{1, \ldots, c\} \\ |J|=s}} \left[ h^0 \left(\Omega^1_{\PP^n} \left(- \sum_{i \in J} d_i +k\right)\right)- h^1 \left(\Omega^1_{\PP^n} \left(- \sum_{i \in J} d_i +k\right)\right) \right].
\end{align*}
\end{proof}
\end{proposition}

\section{Existence of Foliations and the Poincaré Problem}\label{Sec:Existence}
Let $X=V(f_1, \ldots, f_c) \subset \PP^n$ be a smooth complete intersection surface of type 
$(d_1, \ldots, d_c)$ in $\PP^n$. We begin with the case where $n=3$, 
this means that $X=V(f)\subset \PP^3$ is a smooth hypersurface of degree $d_1$.

\begin{proposition}\label{Prop:neq3new}
    Let $X=V(f) \subset \PP^3$ be a smooth hypersurface of degree $d_1$. Then $X$ admits a foliation with tangent sheaf $i^{\ast}\mathcal{O}_{\PP^3}(1-d)$ if and only if 
$$d \geq d_1-1.$$ 
Moreover, 
the dimension of the space of sections defining them is given by
    \begin{equation*}
 h^0 \left(X, \Theta_X \otimes i^{\ast}\mathcal{O}_{\PP^3}(d-1)\right)
    = \begin{cases}
     T(d,d_1) + 1, & \text{if } d =-3+ 2d_1,\\
     T(d,d_1), & \text{otherwise,}
    \end{cases}
\end{equation*}
where 
\begin{align*}
    T(d,d_1)=&\binom{5+d-d_1}{2}\binom{2+d-d_1}{1}-\binom{5+d-2d_1}{2}\binom{2+d-2d_1}{1}\\
    &-\binom{6+d-2d_1}{3}+\binom{6+d-3d_1}{3}.
\end{align*}
\begin{proof}
By \eqref{eq:FormTanCot} we have that
$$H^0 \left(X, \Theta_X \otimes i^{\ast}\mathcal{O}_{\PP^3}(d-1)\right) \cong H^0\left(X,\Omega^1_X \otimes i^{\ast}\mathcal{O}_{\PP^3}(3+d-d_1)\right).$$
We will compute the dimension of the latter space:

Tensor the dual sequence of \eqref{eq:NormalSequence} with the sheaf $i^{\ast}\mathcal{O}_{\PP^3}(3+d-d_1)$ to obtain the exact sequence in $X$
     $$0  \to  H^0(i^{\ast} \mathcal{O}_{\PP^3}(3+d-2d_1)) \to H^0(i^{\ast} \Omega_{\PP^3}^{1}(3+d-d_1)) \to  H^0(\Omega_X^{1}\otimes i^{\ast}\mathcal{O}_{\PP^3}(3+d-d_1)) \to 
     0,$$
     as $ h^1(i^{\ast} \mathcal{O}_{\PP^3}(3+d-2d_1)) = 0 $, 
     by Proposition \ref{Prop:RestrictedBottFirst}. Hence
     \begin{equation}\label{eq:HypP3}
         h^0(\Omega_X^{1}\otimes i^{\ast}\mathcal{O}_{\PP^3}(3+d-d_1))=h^0(i^{\ast} \Omega_{\PP^3}^{1}(3+d-d_1))-h^0(i^{\ast} \mathcal{O}_{\PP^3}(3+d-2d_1)).
     \end{equation}
     By Proposition \ref{PropHart} we have the short exact sequence
     $$0 \to \Omega^1_{\PP^3}(3+d-2d_1) \to \Omega^1_{\PP^3}(3+d-d_1) \to i_{\ast} i^{\ast} \Omega^1_{\PP^3}(3+d-d_1) \to 0.$$
     Using Lemma \ref{LemmBottFla} we conclude that
     $$h^0( i_{\ast} i^{\ast} \Omega^1_{\PP^3}(3+d-d_1))=-h^0(\Omega^1_{\PP^3}(3+d-2d_1))+h^0( \Omega^1_{\PP^3}(3+d-d_1))+h^1( \Omega^1_{\PP^3}(3+d-2d_1)).$$
     Replacing the dimension above into \eqref{eq:HypP3} we obtain
     \begin{align*}
         h^0(\Omega_X^{1}\otimes i^{\ast}\mathcal{O}_{\PP^3}(3+d-d_1))=&-h^0(\Omega^1_{\PP^3}(3+d-2d_1))+h^0( \Omega^1_{\PP^3}(3+d-d_1))+h^1( \Omega^1_{\PP^3}(3+d-2d_1))\\
         &-h^0(i^{\ast}\mathcal{O}_{\PP^3}(3+d-2d_1) ),
     \end{align*}
     Lemma \ref{LemmBottFla} and Proposition \ref{Prop:RestrictedBottFirst} yield the stated formula.

Furthermore, by Lemma \ref{LemmBottFla}, the terms in $T(d,d_1)$ do not vanish for  $d \geq d_1-1$, hence
\begin{equation}\label{eq:Injectiven3}
    h^0(X,\Theta_X\otimes i^*\mathcal O_{\PP^3}(d-1))>0
\textup{ if and only if }
d\geq d_1-1.
\end{equation}
Since foliations with tangent sheaf
$i^*\mathcal O_{\PP^3}(1-d)$
correspond to elements in $
H^0(X,\Theta_X\otimes i^*\mathcal O_{\PP^3}(d-1))
$ up to scalar multiplication,
the previous computation shows that such foliations exist if and only if $d\geq d_1-1$.
\end{proof}
\end{proposition}

For higher-codimensional complete intersection surfaces, we have the following Theorem
\begin{theorem}\label{Theorem:ExistenceSurfaces}
Let $X=V(f_1, \ldots, f_c) \subset \PP^n$, with $n \geq 4$, be a smooth complete intersection surface
of type $(d_1,\ldots,d_c)$ in $\PP^n$, and set
$\lambda=\sum_{j=1}^c d_j$.
Let $d \geq 0$ be an integer.

Then $X$ admits a foliation with tangent sheaf $i^{\ast}\mathcal{O}_{\PP^n}(1-d)$ if and only if 
$$d \geq \lambda - n + 2,$$
except possibly in the case $d=\lambda-n$. In this exceptional case, the existence of such foliations is not determined by the arguments of the proof.

Whenever such foliations exist, the dimension of the space of sections defining them is given by
$$ h^0 \left(X, \Theta_X \otimes i^{\ast}\mathcal{O}_{\PP^n}(d-1)\right)
= S(n,d,d_1, \ldots, d_c)
+\sum_{s=1}^c (-1)^{s+1} \ h^1(G_s^1(d+n-\lambda)),
$$
where
\begin{align*}
    S(n,d,d_1, \ldots, d_c) :=& \sum_{J \subset \{1, \ldots, c\}} (-1)^{|J|} \Big[
\binom{-\sum_{i\in J} d_i + d + n - \lambda - 1}{1}
\binom{-\sum_{i\in J} d_i + d + 2n - \lambda - 1}{n-1}\\
&-
\sum_{j=1}^c
\binom{-\sum_{i\in J} d_i - d_j + d + 2n - \lambda}{n} \Big], \; \textrm{ and }\\
h^1(G^1_{s}(d+n-\lambda)):=&\sum_{\substack{ J \subset \{1, \ldots, c\} \\
 |J|=s}}
h^1\left( \PP^n,\Omega^1_{\PP^n} \left( - \sum_{i \in J} d_i +d+n-\lambda\right) \right)\\
=& \# \left\{ J \subset \{1, \ldots, c \} \Big|  |J|=s, \sum_{i \in J} d_i=d+n-\lambda \right\},
 \end{align*}
 where the last equality follows from Lemma \ref{LemmBottFla}.
\begin{proof}
By \eqref{eq:FormTanCot},
$$H^0 \left(X, \Theta_X \otimes i^{\ast}\mathcal{O}_{\PP^n}(d-1)\right) \cong H^0\left(X,\Omega^1_X \otimes i^{\ast}\mathcal{O}_{\PP^n}(d+n-\lambda)\right).$$
We will compute the dimension of the latter space:

Tensor the dual sequence of \eqref{eq:NormalSequence} with
$i^{\ast}\mathcal{O}_{\PP^n}(d+n-\lambda)$ to obtain the following exact sequence:
$$
0 \to \bigoplus_{j=1}^c i^{\ast} \mathcal{O}_{\PP^n}(-d_j+d+n-\lambda)  
 \to i^{\ast}\Omega^1_{\PP^n}(d+n-\lambda) 
 \to \Omega^1_X \otimes i^{\ast} \mathcal{O}_{\PP^n}(d+n-\lambda) \to 0.
$$
Since $\sum_{j=1}^c h^1(X,  i^{\ast} \mathcal{O}_{\PP^n}(-d_j+d+n-\lambda))=0$ by Proposition \ref{Prop:RestrictedBottFirst}, 
its
associated cohomology sequence 
shows that 
\begin{equation}\label{GoalSeq2}
    h^0(\Omega^1_X \otimes i^{\ast} \mathcal{O}_{\PP^n}(d+n-\lambda))=h^0(i^{\ast}\Omega^1_{\PP^n}(d+n-\lambda))-\sum_{j=1}^c h^0(i^{\ast} \mathcal{O}_{\PP^n}(-d_j+d+n-\lambda)).
\end{equation}
Its second term is computed through Proposition \ref{Prop:RestrictedBottFirst}:
$$
\sum_{j=1}^c h^0(X,  i^{\ast} \mathcal{O}_{\PP^n}(-d_j+d+n-\lambda))=\sum_{j=1}^c \sum_{J \subset \{1, \ldots, c\}}(-1)^{|J|} \ \binom{-\sum_{i \in J}d_i -d_j+d+2n-\lambda}{n}.
$$
To compute the first term of \eqref{GoalSeq2}, consider
$$0 \to \Omega^1_{\PP^n}(d+n-\lambda) \otimes \mathcal{I}_X \to \Omega^1_{\PP^n}(d+n-\lambda) \to i_{\ast} i^{\ast} \Omega^1_{\PP^n}(d+n-\lambda) \to 0.$$
Since $d+n-\lambda \neq 0$, Lemma \ref{LemmBottFla}
gives the following part of
its associated cohomology sequence:
\begin{equation*}
    \begin{split}
        0 \to H^0(\Omega^1_{\PP^n}(d+n-\lambda) \otimes \mathcal{I}_X) \to H^0(\Omega^1_{\PP^n}(d+n-\lambda)) \to H^0(i_{\ast} i^{\ast} \Omega^1_{\PP^n}(d+n-\lambda)) \to&\\
        \to H^1(\Omega^1_{\PP^n}(d+n-\lambda) \otimes \mathcal{I}_X) \to 0,&
    \end{split}
\end{equation*}
from which we conclude that
\begin{align*}
    h^0(X, i^{\ast} &\Omega^1_{\PP^n}(d+n-\lambda))\\
    =& h^1(\Omega^1_{\PP^n}(d+n-\lambda) \otimes \mathcal{I}_X)-h^0(\Omega^1_{\PP^n}(d+n-\lambda) \otimes \mathcal{I}_X)+h^0(\Omega^1_{\PP^n}(d+n-\lambda))\\
    =& \sum_{s=0}^c (-1)^{s} \ h^0 \left(\Omega^1_{\PP^n} \left(- \sum_{i \in J} d_i +d+n-\lambda \right)\right)+\sum_{s=1}^c (-1)^{s+1} \ h^1 \left(\Omega^1_{\PP^n} \left(- \sum_{i \in J} d_i +d+n-\lambda\right)\right),
\end{align*}
where the last equality follows from Proposition \ref{Prop:h0Minush1p1}.

Substituting these identities into \eqref{GoalSeq2}, 
taking note that the values of the first sum above can be computed through Lemma \ref{LemmBottFla},
we finally obtain that
\begin{align*}
    h^0(X&,\Omega^1_X \otimes i^{\ast}\mathcal{O}_{\PP^n}(d+n-\lambda))\\
    =&h^0(i^{\ast}\Omega^1_{\PP^n}(d+n-\lambda))-\sum_{j=1}^c h^0(i^{\ast}\mathcal{O}_{\PP^n}(-d_j+d+n-\lambda))\\
    =&\sum_{J\subset \{1,\ldots,c\}} (-1)^{|J|} \ \Big[ \binom{-\sum_{i\in J}d_i+d+2n-\lambda-1}{n-1}\binom{-\sum_{i\in J}d_i+d+n-\lambda-1}{1}\\
    & -\sum_{j=1}^c  \binom{-\sum_{i\in J}d_i-d_j+d+2n-\lambda}{n} \Big] +\sum_{s=1}^c (-1)^{s+1} \ h^1(G_s^1(d+n-\lambda)),
\end{align*}
where 
$$h^1(G_s^1(d+n-\lambda))= h^1 \left(\Omega^1_{\PP^n} \left(- \sum_{i \in J} d_i +d+n-\lambda\right)\right),$$
which is the formula from the statement.

Finally, Proposition \ref{LemmBottFla} shows that
every summand in the expression above vanishes for
$d<\lambda-n+2$, whereas for
$d\geq \lambda-n+2$
the contribution corresponding to
$J=\varnothing$
is nonzero. Hence
\begin{equation}\label{eq:Injective}
    h^0(X,\Theta_X\otimes i^*\mathcal O_{\PP^n}(d-1))>0
\textup{ if and only if }
d\geq \lambda-n+2,
\end{equation}
except possibly in the case $d=\lambda-n$. Since foliations with tangent sheaf
$i^*\mathcal O_{\PP^n}(1-d)$
correspond to nonzero sections of $
H^0(X,\Theta_X\otimes i^*\mathcal O_{\PP^n}(d-1))
$ up to scalar multiplication, the result follows
\end{proof}
\end{theorem}

\begin{corollary}\label{PPSforCISurfaces}
Let $X=V(f_1, \ldots, f_c) \subset \PP^n$, with $n \geq 3$, be a smooth complete intersection surface
of type $(d_1,\ldots,d_c)$.
If $X$ is invariant by a foliation by curves with tangent 
sheaf $i^{\ast}\mathcal{O}_{\PP^n}(1-d)$, then
$$
\deg(X)\leq \left(\frac{d+n-2}{c}\right)^c.
$$
Moreover, when $n=3$, the converse also holds.
\begin{proof}
Recall that the degree of
$X$
is given by $\deg(X)=\prod_{i=1}^c d_i$. Hence, by the arithmetic-geometric mean inequality,
$$
\lambda = \sum_{i=1}^c d_i \geq  c\,\sqrt[c]{\deg(X)}.
$$
Since $X$ is invariant by a foliation with tangent sheaf
$i^{\ast}\mathcal{O}_{\PP^n}(1-d)$, Theorem
\ref{Theorem:ExistenceSurfaces} (or Proposition
\ref{Prop:neq3new} when $n=3$) yields
\[
d+n-2 \geq \lambda \geq c\,\sqrt[c]{\deg(X)},
\]
and the claimed inequality follows.

If $n=3$, then $c=1$ and the above inequality becomes
\[
\deg(X)\leq d+1.
\]
The converse follows directly from Proposition
\ref{Prop:neq3new}.
\end{proof}
\end{corollary}

\section{Degree of the Singular Scheme}\label{Sec:Degree}
Let $X$ be a compact complex surface and let
$s \in H^0(X, \Theta_X \otimes \mathscr{L}^{\vee})$ be a section with isolated singularities. The section $s$ induces a foliation $\mathscr{F}=[s]$ on $X$. If $q \in \mathrm{Sing}(\mathscr{F})$, its multiplicity is defined by
$$
\mu_q(\mathcal{F})
=
\dim_{\CC}
\left(
\frac{\mathcal{O}_{X,q}}
{(a_i,b_i) \cdot \mathcal{O}_{X,q}}
\right),
$$
where \(a_i\) and \(b_i\) are the local coefficients of a vector field defining
\(\mathcal{F}\) near \(q\).

By \cite[Proposition 2.4]{RuledSurfaces}, the number of singular points of $s$, counted with multiplicities, is given by
$$
\sum_{q \in \mathrm{Sing}(\mathscr{F})} \mu_q(\mathscr{F})=c_2(\Theta_X \otimes \mathscr{L}^{\vee})[X].
$$
We shall refer to this number as the \emph{degree} of the singular scheme $Z_s$ of $s$, and will be denoted by $\mathrm{deg}(Z_s)$.

\medskip
\begin{theorem}\label{eq:NumberofSing}
Let $X = V(f_1,\ldots,f_c) \subset \PP^n$ be a smooth complete intersection surface of type $(d_1, \ldots, d_c)$ in $\PP^n$ and set $\lambda=\sum_{j=1}^c d_j$.
Then, for any section 
$s \in H^0(X, \Theta_X \otimes i^*\mathcal{O}_{\PP^n}(d-1))$ with isolated singularities, the degree of its singular scheme $Z_s$ is
$$\mathrm{deg}(Z_s)=\prod_{i=1}^c d_i
\left[
\binom{n+1}{2}
+\sum_{i<j} d_i d_j
+ \sum_{i=1}^c d_i^2
-(n+d)\sum_{i=1}^c d_i
+(n+1)(d-1)
+(d-1)^2
\right].$$
Moreover, $ \mathrm{deg}(Z_s) > 0 $ for 
$d \geq \sum_{i=1}^c d_i-n+2 = \lambda-n+2$, 
and hence every foliation with isolated singularities
as in Theorem \ref{Theorem:ExistenceSurfaces}
has a nonempty singular scheme.

\end{theorem}

\begin{proof} 
The proof of the last statement is elementary: denote by $ P(d) $ the polynomial
between brackets in the expression above: one easily sees that $ P(\lambda-n+2) > 0 $
and moreover, that its derivative $ P'(d) $ is positive as well, for 
$ d \geq \lambda-n+2 $. The conclusion follows.
Now we come to the computation of the value of $ \mathrm{deg}(Z_s) $:

Let $\mathscr{L} = i^*\mathcal{O}_{\PP^n}(1-d)$ and $H := c_1(\mathcal{O}_X(1)) = c_1(i^*\mathcal{O}_{\PP^n}(1))$. Since $\Theta_X$ has rank $2$, applying $V = \Theta_X$ and $L = \mathscr{L}^{\vee}$ to equation (2.6) in \cite[p.~28]{Friedman}, we obtain
\begin{equation}\label{eq:SecChernClass_thm}
c_2(\Theta_X \otimes \mathscr{L}^{\vee})
=
c_2(\Theta_X) + c_1(\Theta_X)c_1(\mathscr{L}^{\vee}) + c_1(\mathscr{L}^{\vee})^2.
\end{equation}
By the canonical bundle formula \eqref{CanonicalBundle},
$$
c_1(\Theta_X) = \left(n+1 - \sum_{i=1}^c d_i\right)H.
$$
Moreover, the normal exact sequence \eqref{eq:NormalSequence}
and the Whitney sum formula (see \cite[p.~51]{Fulton}) yield
$$
c_2(\Theta_X)
=
c_2(i^*\Theta_{\PP^n})
- c_1(\Theta_X)\, c_1\!\left(\bigoplus_{i=1}^c i^*\mathcal{O}_{\PP^n}(d_i)\right)
- c_2\!\left(\bigoplus_{i=1}^c i^*\mathcal{O}_{\PP^n}(d_i)\right).
$$
Using
$$
c_1(i^*\Theta_{\PP^n}) = (n+1)H,
\quad
c_2(i^*\Theta_{\PP^n}) = \binom{n+1}{2}H^2,
$$
and
$$
c_1\!\left(\bigoplus_{i=1}^c i^*\mathcal{O}_{\PP^n}(d_i)\right)
=
\left(\sum_{i=1}^c d_i\right)H,
$$
$$
c_2\!\left(\bigoplus_{i=1}^c i^*\mathcal{O}_{\PP^n}(d_i)\right)
=
\left(\sum_{i<j} d_i d_j\right)H^2,
$$
we obtain
\begin{align*}
c_2(\Theta_X)
&=
\binom{n+1}{2}H^2
-
\left(n+1-\sum_{i=1}^c d_i\right)\left(\sum_{i=1}^c d_i\right)H^2
-
\left(\sum_{i<j} d_i d_j\right)H^2 \\
&=
\left[
\binom{n+1}{2}
-(n+1)\sum_{i=1}^c d_i
+\left(\sum_{i=1}^c d_i\right)^2
-\sum_{i<j} d_i d_j
\right] H^2.
\end{align*}
Since
$$
\left(\sum_{i=1}^c d_i\right)^2
=
\sum_{i=1}^c d_i^2 + 2\sum_{i<j} d_i d_j,
$$
it follows that
$$
c_2(\Theta_X)
=
\left[
\binom{n+1}{2}
-(n+1)\sum_{i=1}^c d_i
+ \sum_{i=1}^c d_i^2
+ \sum_{i<j} d_i d_j
\right] H^2.
$$
Since $c_1(\mathscr{L}^{\vee}) = (d-1)H$, substituting into \eqref{eq:SecChernClass_thm} gives
$$
c_2(\Theta_X \otimes \mathscr{L}^{\vee})
=
\left[
\binom{n+1}{2}
+\sum_{i<j} d_i d_j
+ \sum_{i=1}^c d_i^2
-(n+d)\sum_{i=1}^c d_i
+(n+1)(d-1)
+(d-1)^2
\right] H^2.
$$
Finally, since $X$ is a complete intersection surface, we have
$$
H^2[X] = \deg(X) = \prod_{i=1}^c d_i
$$
(see \cite[p.~369]{MarcioAnn}), and the result follows.
\end{proof}

\section{Foliations and their singular schemes}\label{Sec:SingScheme}


Let $X=V(f_1, \ldots, f_c) \subset \PP^n$ be a smooth complete intersection surface of type $(d_1, \ldots, d_c)$.
As in Theorem \ref{Theorem:ExistenceSurfaces}, let
$\lambda=\sum_{j=1}^c d_j$, let $d$ be an integer such that $d\geq \lambda-n+2$
and consider a nonzero section
$s \in H^0(X, \Theta_X \otimes i^{\ast} \mathcal{O}_{\PP^n}(d-1))$   
with isolated singularities, and singular scheme $Z=Z_s$.
In this section we will prove Theorem \textbf{B} from the Introduction
(Theorem \ref{Th:RigiditySurfaces} below), which is the effective version of 
{\cite[Theorem 2.2]{CampilloOlivares}} for this class of manifolds $ X $ and sections
$s$.

The Koszul complex associated to $s$, see \cite[p.76]{RRAlgebra}, provides the following
resolution of the ideal sheaf $\mathcal{I}_Z$ of $ Z $:
\begin{equation*}
    0 \longrightarrow \bigwedge^2 \left(\Omega^1_X \otimes i^{\ast} \mathcal{O}_{\PP^n}(1-d) \right)  \overset{i_s}{\longrightarrow} \Omega^1_X \otimes i^{\ast} \mathcal{O}_{\PP^n}(1-d) \overset{i_s}{\longrightarrow} \mathcal{I}_Z \longrightarrow 0,
\end{equation*}
where $i_s$ denotes contraction by the section $s$.
Since $\omega_X \cong i^{\ast} \mathcal{O}_{\PP^n}\left(-n-1+\sum_{j=1}^c d_j\right)$ by \eqref{CanonicalBundle}, the Koszul complex
can be rewritten as
\begin{equation*}
    0 \longrightarrow i^{\ast} \mathcal{O}_{\PP^n}(\lambda -n-1)   \otimes i^{\ast} \mathcal{O}_{\PP^n}(1-d)^{\otimes 2} \overset{i_s}{\longrightarrow} \Omega^1_X \otimes i^{\ast} \mathcal{O}_{\PP^n}(1-d) \overset{i_s}{\longrightarrow} \mathcal{I}_Z \longrightarrow 0.
\end{equation*}
Tensoring it with $\Theta_X \otimes i^{\ast} \mathcal{O}_{\PP^n}(d-1)$ gives the exact
sequence
\begin{equation*}
    0 \longrightarrow \Theta_X \otimes i^{\ast} \mathcal{O}_{\PP^n}(\lambda -d-n)   
    \overset{i_s \otimes \mathrm{id}}{\longrightarrow}
    \mathrm{End}(\Theta_X) \overset{i_s \otimes \mathrm{id}}{\longrightarrow} \Theta_X \otimes i^{\ast} \mathcal{O}_{\PP^n}(d-1) \otimes \mathcal{I}_Z \longrightarrow 0,
\end{equation*}
whose cohomology sequence
\begin{equation}\label{fundamental}
\begin{array}{c@{\hspace{1pt}}c@{\hspace{1pt}}c@{\hspace{1pt}}c@{\hspace{1pt}}  c@{\hspace{1pt}}c@{\hspace{1pt}}c@{\hspace{1pt}}c}
 0 & \to & H^0(X,\Theta_X \otimes i^{\ast} \mathcal{O}_{\PP^n}(\lambda -d-n)) & \to &  H^0(X, \mathrm{End}(\Theta_X)) & 
  \overset{(i_s \otimes \mathrm{id})^0}{\longrightarrow} & 
    H^0(X, \Theta_X\otimes i^\ast\mathcal O_{\PP^n}(d-1)\otimes\mathcal I_Z) & \to \\
     & \to & H^1(X, \Theta_X \otimes i^{\ast} \mathcal{O}_{\PP^n}(\lambda -d-n)) & \to &   H^1(X, \mathrm{End}(\Theta_X)) & \longrightarrow & \cdots & 
    \end{array}
\end{equation}
is fundamental for what follows:

\begin{remark}\label{Remark:CharSections}
Note that
$ H^0(X, \Theta_X\otimes i^\ast\mathcal O_{\PP^n}(d-1)\otimes\mathcal I_Z) $
is the space of sections 
$s'\in H^0(X,\Theta_X\otimes i^\ast\mathcal O_{\PP^n}(d-1))$ that vanish along 
$Z = Z_s $, that
is, sections whose singular scheme satisfies $ Z_{s'} \supset Z_s $. 
Note that $ s' $ may or not have isolated singularities and
moreover that, by construction,
the map $ (i_s \otimes \mathrm{id})^0 $ from \eqref{fundamental} is given by
$ (i_s\otimes\mathrm{id})^0(\phi)=\phi(s) $, for 
$\phi \in H^0(X, \mathrm{End}(\Theta_X))$. Hence, this map is surjective if and only
if every nonzero section 
$ s' \in  H^0(X, \Theta_X\otimes i^\ast\mathcal O_{\PP^n}(d-1)\otimes\mathcal I_Z) $
is of 
the form $s'=\phi(s)$, for some global endomorphism 
$\phi \in H^0(X, \mathrm{End}(\Theta_X))$.

Assume that $ (i_s\otimes\mathrm{id})^0(\phi)=\phi(s) $ is surjective: 

In case $ \Theta_X $ is \textit{simple} (which means that every such $ \phi $
has the form $\phi=\lambda \cdot \mathrm{id}_{\Theta_X}$ for some $\lambda\in\CC$),
one concludes that every nonzero section 
$ s' $ such that $ Z_{s'} \supset Z_s $
satisfies $s'= \lambda \cdot s$. 
This means that $ s $ and $ s' $ define the same foliation $ [s] $ and hence
that this is the \textit{unique} foliation having $ Z_s $ as singular scheme.
That is,
that the foliation $ [s] $ is uniquely determined by its singular scheme. 

In case $ \Theta_X $ is not simple, there may exist foliations $ [s'] \neq [s] $
such that $ Z_{s'} = Z_s $, for some
$ s' \in  H^0(X, \Theta_X\otimes i^\ast\mathcal O_{\PP^n}(d-1)\otimes\mathcal I_Z) $:
different foliations with isolated singularities that share the same singular scheme.
Indeed, any \emph{invertible} global endomorphism $\phi\in H^0(X,\mathrm{End}(\Theta_X))$ induces a section $\phi(s)$ with $Z_{\phi(s)}=Z_s$, and 
whenever $ \phi \neq \lambda \cdot \mathrm{id}_{\Theta_X} $, the foliation
$[\phi(s)]$ is different from $[s]$. 
See Theorem \ref{Thm:ESSifInvertible} below.
\end{remark}
From \eqref{fundamental}, the morphism $(i_s \otimes \mathrm{id})^0$ is surjective
for those values of $ d $ for which
\begin{equation}\label{eq:GoalSingularSchemeSurfaces}
    H^1(X, \Theta_X \otimes i^{\ast} \mathcal{O}_{\PP^n}(\lambda -d-n))=0,
\end{equation}
These values do exist, by {\cite[Theorem 2.2]{CampilloOlivares}}. The content of
Theorem \ref{Th:RigiditySurfaces} below is the explicit computation 
of them.

We start the computation in the case of smooth hypersurfaces and then extend the argument to arbitrary smooth complete intersection surfaces. 

The following lemma 
proves \eqref{eq:GoalSingularSchemeSurfaces} for
$n=3$ and will be incorporated into the proof of
Theorem \ref{Th:RigiditySurfaces}.
\begin{lemma}\label{Lemma:n3}
Let $X=V(f) \subset \PP^3$ be a smooth hypersurface of degree $d_1$, then
\begin{equation*}
    h^1(X, \Theta_X \otimes i^{\ast} \mathcal{O}_{\PP^3}(d_1-d-3))=0, \quad \textup{for } d >2 d_1-3.
\end{equation*}
\begin{proof}
Tensoring the normal sequence \eqref{eq:NormalSequence} with 
$i^{\ast}\mathcal{O}_{\PP^3}(d_1-d-3)$ gives
\begin{equation*}
    0 \to \Theta_X \otimes i^{\ast}\mathcal{O}_{\PP^4}(i^{\ast}\mathcal{O}_{\PP^3}(d_1-d-3))  \to i^{\ast}\Theta_{\PP^3}(d_1 -d -3) \to i^{\ast}\mathcal{O}_{\PP^3}(2 d_1-d-3) \to 0.
\end{equation*}
Taking cohomology in $X$, we obtain
\begin{equation}\label{SurfCohHyper}
    \cdots \to H^0( i^{\ast}\mathcal{O}_{\PP^3}(2 d_1-d-3)) \to H^1(\Theta_X \otimes i^{\ast}\mathcal{O}_{\PP^3}(d_1 -d -3))  \to  H^1( i^{\ast}\Theta_{\PP^3}(d_1-d -3)) \to \cdots
\end{equation}
By Proposition \ref{Prop:RestrictedBottFirst},
\begin{equation}\label{eq:Dimension1RigiditySurfacen3}
    h^0(i^{\ast} \mathcal{O}_{\PP^3}(2 d_1 -d-3))=0, \quad \textup{for } d > 2 d_1-3.
\end{equation}
For the remaining term, since $\mathcal{I}_X \cong \mathcal{O}_{\PP^3}(-d_1)$ by 
\eqref{KoszCxSh}, then Proposition \ref{PropHart} gives the short exact sequence 
in $\PP^3$
$$0 \to \Theta_{\PP^3}(-d-3) \to \Theta_{\PP^3}(d_1-d-3) \to i_{\ast}i^{\ast}\Theta_{\PP^3}(d_1-d-3) \to 0.$$
The cohomology sequence associated to the previous sequence is given by
$$\cdots \to H^1(\Theta_{\PP^3}(d_1-d-3)) \to H^1(i_{\ast}i^{\ast}\Theta_{\PP^3}(d_1-d-3)) \to H^2(\Theta_{\PP^3}(-d-3)) \to \cdots,$$
where $H^1(\PP^3, i_{\ast}i^{\ast}\Theta_{\PP^3}(d_1-d-3)) \cong H^1(X,i^{\ast}\Theta_{\PP^3}(d_1-d-3))$. By Serre duality and 
Lemma \ref{LemmBottFla} we have that
$$h^1(\Theta_{\PP^3}(d_1-d-3))=h^2(\Omega^1_{\PP^3}(d-d_1-1))=0,$$
and
$$h^{2}( \Theta_{\PP^3}(-d-3)) = h^1(\Omega^1_{\PP^n}(d-1))=\begin{cases}
     1, & \text{if } d=1,\\
     0, & \text{otherwise},
    \end{cases}$$
concluding that
\begin{equation}\label{eq:Dimension2RigiditySurfacen3}
    h^1(i_{\ast} i^{\ast} \Theta_{\PP^3}(d_1-d-3))=0, \quad \textup{for } d \neq 1.
\end{equation}
Combining \eqref{eq:Dimension1RigiditySurfacen3} and \eqref{eq:Dimension2RigiditySurfacen3} with \eqref{SurfCohHyper} yields
\begin{equation*}
    h^1(X, \Theta_X \otimes i^{\ast} \mathcal{O}_{\PP^3}(d_1-d-3))=0, \quad \textup{for } d >2 d_1-3,
\end{equation*}
since $1 \leq 2 d_1-3$.
\end{proof}
\end{lemma}

\medskip
The result for $X$ of arbitrary codimension is presented in the following theorem:
\begin{theorem}\label{Th:RigiditySurfaces}
Let $ X=V(f_1, \ldots, f_c) \subset \PP^n $, $n \geq 3$, 
be a smooth complete intersection surface of type $(d_1, \ldots, d_c)$ in $\PP^n$ with $1 < d_1 \leq \cdots \leq d_c$, and set 
$\mathscr{L}_d=i^{\ast} \mathcal{O}_{\PP^n}(1-d)$ with $d \geq \lambda -n+2$.

Let $ \mathscr{F} = [s] \in \mathscr{F}ol(\mathscr{L}_d, X ) $ be a foliation with isolated singularities and satisfying that 
$$ d > \lambda + d_c-n.$$
Then every nonzero section 
$ s' \in H^0(X, \Theta_X \otimes i^{\ast} \mathcal{O}_{\PP^n}(d-1)) $ 
satisfies that $Z_{s'} \supset Z_s$ if and only if 
there exists a unique global endomorphism $\phi$ of $\Theta_X$ such that $s'=\phi(s)$.
\begin{proof}
The case $n=3$ follows from Lemma \ref{Lemma:n3}. Hence, assume that $n\ge4$.

Tensoring the normal sequence \eqref{eq:NormalSequence} with $i^{\ast}\mathcal{O}_{\PP^n}(\lambda-d-n)$ to obtain the sequence on $X$
\begin{equation*}
    0 \to \Theta_X \otimes i^{\ast}\mathcal{O}_{\PP^n}(\lambda -d -n)  \to i^{\ast}\Theta_{\PP^n}(\lambda -d -n) \to \bigoplus_{j=1}^c i^{\ast}\mathcal{O}_{\PP^n}(\lambda +d_j-d-n) \to 0.
\end{equation*}
Its associated cohomology sequence contains
\begin{equation}\label{SurfCoh}
    \bigoplus_{j=1}^c H^0( i^{\ast}\mathcal{O}_{\PP^n}(\lambda +d_j-d-n)) \to H^1(\Theta_X \otimes i^{\ast}\mathcal{O}_{\PP^n}(\lambda -d -n))  \to  H^1( i^{\ast}\Theta_{\PP^n}(\lambda -d -n)).
\end{equation}
By Proposition \ref{Prop:RestrictedBottFirst}
$$h^0( i^{\ast}\mathcal{O}_{\PP^n}(\lambda +d_{j}-d-n))=0,$$
whenever $d > \lambda +d_{j} -n$. Since $d_j \leq d_c$, our hypothesis implies
\begin{equation}\label{eq:Dimension1RigiditySurfaces}
    \sum_{j=1}^c h^0(i^{\ast}\mathcal{O}_{\PP^n}(\lambda +d_j-d-n))=0, \quad \textup{for } d > \lambda +d_c -n.
\end{equation}
For the remaining term, by Proposition
$$0 \to \Theta_{\PP^n}(\lambda -d-n) \otimes \mathcal{I}_X \to \Theta_{\PP^n}(\lambda -d-n)\to i_{\ast}i^{\ast}\Theta_{\PP^n}(\lambda -d-n) \to 0,$$
whose associated sequence in cohomology contains
$$H^1( \Theta_{\PP^n}(\lambda-d-n))  \to  H^1( i_{\ast} i^{\ast} \Theta_{\PP^n}(\lambda-d-n) ) \to  H^2( \Theta_{\PP^n}(\lambda-d-n) \otimes \mathcal{I}_X)
     \to H^2( \Theta_{\PP^n}(\lambda-d-n))$$
Using Serre duality together with Lemma \ref{LemmBottFla}, we conclude that
\begin{align*}
    &h^1(\Theta_{\PP^n}(\lambda -d-n))=h^{n-1}(\Omega^1_{\PP^n}(d-\lambda-1))=0, \textup{ and}\\
    &h^2(\Theta_{\PP^n}(\lambda -d-n))=h^{n-2}(\Omega^1_{\PP^n}(d-\lambda-1))=0,
\end{align*}
since $n \geq 4$. Consequently
$$H^1(i_{\ast} i^{\ast} \Theta_{\PP^n}(\lambda-d-n) ) \cong H^2( \Theta_{\PP^n}(\lambda-d-n) \otimes \mathcal{I}_X).$$
To establish the vanishing of the latter group, tensor the first sequence in \eqref{eq:SystemIdeal} with $\Theta_{\PP^n} (\lambda -d-n)$ to obtain
\begin{equation}\label{eq:ProofRigiditySurfaces1}
    0 \to \mathrm{Im}\ \varphi_2 \otimes \Theta_{\PP^n} (\lambda -d-n) 
    \to \bigoplus_{j=1}^c \Theta_{\PP^n}(\lambda-d-n-d_j) 
    \to \Theta_{\PP^n} (\lambda -d-n) \otimes \mathcal{I}_X \to 0.
\end{equation}
By Serre duality and Lemma \ref{LemmBottFla}, we have the vanishings
$$h^2(\Theta_{\PP^n}(\lambda -d-n-d_j))=h^{n-2}(\Omega^1_{\PP^n}(d+d_j-\lambda -1))=0$$
for each $j=1, \ldots, c$, since $n \geq 4$. Therefore, the cohomology sequence associated to \eqref{eq:ProofRigiditySurfaces1} contains
\begin{equation}\label{eq:ProofRigiditySurfaces2}
    0 \to H^2(\Theta_{\PP^n} (\lambda -d-n) \otimes \mathcal{I}_X) \to H^3(\mathrm{Im} \ \varphi_2 \otimes \Theta_{\PP^n} (\lambda -d-n)) \to \bigoplus_{j=1}^c H^3( \Theta_{\PP^n}(\lambda-d-n-d_j))
\end{equation}
Replacing $k=2, \ldots, n-4$ in Lemma \ref{Lemma:Claim} below, we conclude the isomorphisms
$$H^3(\mathrm{Im}\ \varphi_{2} \otimes \Theta_{\PP^n} (\lambda -d-n)) \cong H^{n-2}(\mathrm{Im}\ \varphi_{c-1} \otimes \Theta_{\PP^n} (\lambda -d-n)).$$
Finally, tensoring the last sequence in \eqref{eq:SystemIdeal} with $\Theta_{\PP^n} (\lambda -d-n)$ to obtain
\begin{equation}\label{eq:ProofRigiditySurfaces3}
\begin{split}
    0 \to \Theta_{\PP^n} (-d-n) 
    \to \bigoplus_{\substack{ J \subset \{1, \ldots, c\} \\ |J|=c-1}} \Theta_{\PP^n} &\left(- \sum_{i \in J} d_i + \lambda -d-n \right)
    \to\\
    & \to \mathrm{Im}\ \varphi_{c-1} \otimes \Theta_{\PP^n} (\lambda -d-n) \to 0.
\end{split}
\end{equation}
Apply Serre duality and Lemma \ref{LemmBottFla} to obtain the vanishings 
$$h^{n-2}\left( \Theta_{\PP^n} \left(-\sum_{i \in J}d_i+\lambda -d-n \right) \right)=h^{2}\left( \Omega^1_{\PP^n} \left(\sum_{i \in J}d_i+d- \lambda-1 \right) \right)=0$$
for all $J \subset \{1, \ldots, c\}$ with $|J|=c-1$. Therefore, the cohomology sequence associated to \eqref{eq:ProofRigiditySurfaces3} is given by 
\begin{align*}
    0 \to H^{n-2}(\mathrm{Im}\ \varphi_{c-1} \otimes \Theta_{\PP^n} (\lambda -d-n))& \to H^{n-1}(\Theta_{\PP^n}(-d-n)) \to \\
    & \to \bigoplus_{\substack{ J \subset \{1, \ldots, c\} \\ |J|=c-1}} H^{n-1} \left(\Theta_{\PP^n} \left(- \sum_{i \in J} d_i + \lambda -d-n \right) \right) \to \cdots.
\end{align*}
By Serre duality and Lemma \ref{LemmBottFla}
$$h^{n-1}( \Theta_{\PP^n}(-d-n)) = h^1(\Omega^1_{\PP^n}(d-1))=\begin{cases}
     1, & \text{if } d=1,\\
     0, & \text{otherwise}.
    \end{cases}$$
Hence 
\begin{equation*}
    h^{3}(\mathrm{Im} \ \varphi_{2} \otimes \Theta_{\PP^n}(\lambda-d-n))=h^{n-2}(\mathrm{Im} \ \varphi_{c-1} \otimes \Theta_{\PP^n}(\lambda-d-n))=0, \quad \textup{for } d \neq 1.
\end{equation*}
Substituting this into \eqref{eq:ProofRigiditySurfaces2}, we conclude that
\begin{equation}\label{eq:Dimension2RigiditySurfaces}
    h^1(i_{\ast}i^{\ast} \Theta_{\PP^n}(\lambda-d-n)) =h^2(\Theta_{\PP^n}(\lambda-d-n) \otimes \mathcal{I}_X)=0, \quad \textup{for } d \neq 1.
\end{equation}
Combining the vanishings \eqref{eq:Dimension1RigiditySurfaces} and \eqref{eq:Dimension2RigiditySurfaces} with \eqref{SurfCoh}, we conclude that 
$$
H^1(X,\Theta_X \otimes i^{\ast}\mathcal{O}(\lambda-d-n))=0, \quad 
 \textup{for } d>\lambda +d_c-n,
$$ 
since $1<\lambda+d_c-n$. 

It follows from \eqref{eq:GoalSingularSchemeSurfaces} that
the morphism $(i_s \otimes \mathrm{id})^0$ in \eqref{fundamental} is surjective
for these values of $ d $.

We see that it is also injective from \eqref{fundamental} since,
replacing $d$ for $\lambda-d-n+1$ into \eqref{eq:Injective}, we see that
$H^0(X, \Theta_X \otimes i^{\ast}\mathcal{O}_{\PP^n}(\lambda-d-n))=0$, for 
$d > \lambda +d_c-n$.
It follows that the morphism $(i_s \otimes \mathrm{id})^0$ in \eqref{fundamental} is an isomorphism for $d > \lambda +d_c-n$, concluding the uniqueness of the global endomorphism $\phi$, for $ n\geq 4$.

For $ n=3 $ and $ d > 2 d_1 - 3$, Lemma \ref{Lemma:n3} shows that 
$ (i_s \otimes \mathrm{id})^0 $ is 
surjective and, replacing $d$ for $ d_1 - d - 2 $ into \eqref{eq:Injectiven3}, we
see that $H^0(X, \Theta_X \otimes i^{\ast}\mathcal{O}_{\PP^3}(d_1-d-3))=0$ and
hence, that it is also injective.
\end{proof}
\end{theorem}

Now we prove the claim made during  
the proof of Theorem \ref{Th:RigiditySurfaces}:
\begin{lemma}\label{Lemma:Claim}
Under the hypotheses of Theorem \ref{Th:RigiditySurfaces}:
$$H^{k+1}(\PP^n,\mathrm{Im}\ \varphi_k \otimes \Theta_{\PP^n} (\lambda -d-n)) \cong H^{k+2}(\PP^n,\mathrm{Im}\ \varphi_{k+1} \otimes \Theta_{\PP^n} (\lambda -d-n))$$
for $k=0, \ldots, n-4$.  
\begin{proof}
Tensor the $k$-th short exact sequence in \eqref{eq:SystemIdeal} with the sheaf $\Theta_{\PP^n} (\lambda -d-n)$ to obtain
\begin{align*}
    0 \to \mathrm{Im} \ \varphi_{k+1} \otimes \Theta_{\PP^n} (\lambda -d-n) \to \bigoplus_{\substack{ J \subset \{1, \ldots, c\} \\ |J|=k}} \Theta_{\PP^n} \left(- \sum_{i \in J} d_i + \lambda -d-n \right) \to&\\
    \to \mathrm{Im} \ \varphi_{k} \otimes \Theta_{\PP^n} (\lambda -d-n) \to 0,&
\end{align*}
whose associated cohomology sequence is 
\begin{equation*}
    \begin{split}
    \cdots & \to  \bigoplus_{\substack{ J \subset \{1, \ldots, c\} \\ |J|=k}} H^{k+1}\left( \Theta_{\PP^n} \left(- \sum_{i \in J} d_i + \lambda -d-n \right) \right)  \to   H^{k+1}(\mathrm{Im} \ \varphi_{k} \otimes \Theta_{\PP^n} (\lambda -d-n))  \to\\
         & \to  H^{k+2}(\mathrm{Im} \ \varphi_{k+1} \otimes \Theta_{\PP^n} (\lambda -d-n)) \to  \bigoplus_{\substack{ J \subset \{1, \ldots, c\} \\ |J|=k}} H^{k+2}\left( \Theta_{\PP^n} \left(- \sum_{i \in J} d_i + \lambda -d-n \right) \right)  \to  \cdots
\end{split}
\end{equation*}
By Serre duality, for each selection of $J \subset \{1, \ldots, c \}$ with $|J|=k$
\begin{align*}
    &h^{k+1}\left( \Theta_{\PP^n} \left(- \sum_{i \in J} d_i + \lambda -d-n \right) \right) = h^{n-k-1}\left( \Omega^1_{\PP^n} \left( \sum_{i \in J} d_i - \lambda +d-1 \right) \right) = 0, \textrm{ and} \\
    &h^{k+2}\left( \Theta_{\PP^n} \left(- \sum_{i \in J} d_i + \lambda -d-n \right) \right) = H^{n-k-2}\left( \Omega^1_{\PP^n} \left( \sum_{i \in J} d_i - \lambda +d-1 \right) \right)= 0,
\end{align*}
since $k=0, \ldots, n-4$, concluding the isomorphism of the statement.
\end{proof}
\end{lemma}

The following Corollaries of Theorem \ref{Th:RigiditySurfaces} come from 
Remark \ref{Remark:CharSections} and from 
the well-known fact \cite[Corollary 1.28]{GeomModuliSpaces} 
that stable bundles are 
simple:
\begin{corollary}\label{Cor:EndoHyperP3}
    Let $X=V(f) \subset \PP^3$ be a smooth hypersurface of degree $d_1 \geq 3$. Then, any foliation on $X$ with isolated singularities and tangent sheaf $i^{\ast}\mathcal{O}_{\PP^3}(1-d)$ with $d > 2d_1-3$, is uniquely determined by its singular scheme.
 \begin{proof}
    $\Theta_X$ is stable by \cite[Théorème]{Fah}, for $d_1=3$ and by 
    \cite[Proposition 3.5.1]{PhdThesisLiu}, for $d_1 \geq 4$. 
 \end{proof}
\end{corollary}
 For a different proof, see \cite[Corollary 3.9]{B}.
\begin{corollary}\label{Cor:Endo22}
    Let $X=V(f_1,f_2) \subset \PP^4$ be a smooth complete intersection surface of type $(2,2)$. Then, any foliation on $X$ with isolated singularities and tangent sheaf $i^{\ast}\mathcal{O}_{\PP^3}(1-d)$ with $d > 2$, is uniquely determined by its singular scheme.
\begin{proof}
    $\Theta_X$ is stable by \cite[Théorème]{Fah}.
\end{proof}
\end{corollary}
 Finally, in concern with foliations that have the same singular scheme, we have the
 following result:
\begin{theorem}\label{Thm:ESSifInvertible}
    Let $ X=V(f_1, \ldots, f_c) \subset \PP^n $, $n \geq 3$, 
be a smooth complete intersection surface of type $(d_1, \ldots, d_c)$ in $\PP^n$ with 
$ 1 < d_1 \leq \cdots \leq d_c $, and let 
$ s \in H^0(X, \Theta_X \otimes i^{\ast} \mathcal{O}_{\PP^n}(d-1))$ 
be a nonzero section with isolated singularities and singular scheme $ Z_s $. 
Assume $ d > \lambda + d_c-n $, and let 
$ s' \in H^0(X, \Theta_X \otimes i^{\ast} \mathcal{O}_{\PP^n}(d-1)) $ be a global section
such that that $ Z_{s'} \supset Z_s $ and let 
$ \phi \in H^{0}(X, \mathrm{End}(\Theta_X)) $ be the global endomorphism such that
$ s'= \phi(s) $. If $ \phi $ is invertible then $ Z_{s'}=Z_s $.

 In consequence, every invertible endomorphism $ \phi $ such that 
 $ \phi \neq \lambda \cdot \mathrm{id}_{\Theta_X} $ (if any) gives rise to a
 foliation $ [ s' = \phi( s ) ] \neq [ s ] $ with 
 $ Z_{s'}=Z_s $.
\end{theorem}
\begin{proof}
To simplify the notation, let $ Z' $ and $ Z $ denote $ Z_{s'} $ and $ Z_s $ respectively.
Let $ s' $ be a global section such that $ Z' \supset Z $ and let
$ \phi \in H^{0}(X, \mathrm{End}(\Theta_X)) $ be the invertible endomorphism such
that $ s'= \phi(s) $. Consider an open cover $ \{ U_i\}_{i \in I } $ which trivializes 
the tangent bundle $ \Theta_X $. For each $ i \in I $, the restrictions
of $ s $, $ s' $ and $ \phi $ to $ U_i $ are, respectively, vector fields 
$ X $, $ X' $
and a $2\times 2$ matrix $ \phi_i = \phi|_{U_i} $
satisfying $ \phi_i \cdot X = X' $. This is: 
\begin{equation*}
 \begin{pmatrix}
     a(x) & b(x) \\
     c(x) & d(x)  \end{pmatrix} 
  \begin{pmatrix}
         X_1(x)\\
         X_2(x)
     \end{pmatrix} =
     \begin{pmatrix}
         X'_1(x)\\
         X'_2(x)
     \end{pmatrix}, \quad{ x\in U_i } .  
\end{equation*}
It is easy to see that the equation above is \textit{equivalent} to
\[ {\mcI_{Z'}}_{|_{U_i}} = \langle X'_1(x), X'_2(x) \rangle \subset 
 \langle X_1(x), X_2(x) \rangle = {\mcI_{Z}}_{|_{U_i}}. \]
On the other hand, since $ \phi_i $ is invertible, we also have 
$ \phi_i^{-1} \cdot X' = X $ and hence 
$ {\mcI_{Z'}}_{|_{U_i}} = {\mcI_{Z}}_{|_{U_i}}.$

Since the last equality holds for every $ i\in I $, the conclusion is
that $ Z' = Z $.
\end{proof}

\section{Global Endomorphisms of the Tangent Bundle of X}\label{Sec:GlobalEnd}
Theorem \ref{Th:RigiditySurfaces} shows that the problem of determining whether a foliation is determined by its singular scheme is governed by the space of global endomorphisms of the tangent bundle. We now study some properties of this space, for smooth projective complete intersection surfaces.

Let $\phi \in H^0(X, \mathrm{End}(\Theta_X))$ be a global endomorphism of the tangent bundle. For each point $p \in X$, the induced linear map $\phi_p:T_p X \to T_p X$ has characteristic polynomial 
$$P_p(x)=x^2-\mathrm{tr}(\phi_p)x+\mathrm{det}(\varphi_p),$$
where $\mathrm{tr}(\phi_p), \mathrm{det}(\phi_p) \in H^0(X, \mathcal{O}_X)=\CC$. Therefore, both $\mathrm{tr}(\phi)$ and $\mathrm{det}(\phi_p)$ are constants, and the characteristic polynomial of $\phi_p$ is independent of $p$. We denote it by
\begin{equation}\label{eq:CharPoly}
    P(x)=x^2-\mathrm{tr}(\phi)x+\mathrm{det}(\phi).   
\end{equation}
\begin{proposition}\label{Prop:TXSplits}
    Let $X$ be a smooth connected compact surface. Then there exists a global endomorphism $\phi \in H^0(X, \mathrm{End}(\Theta_X))$ with two distinct eigenvalues, if and only if $\Theta_X \cong L_1 \oplus L_2$, for some line bundles $L_1$ and $L_2$ on $X$.
\begin{proof}
    Assume that $\phi$ has two distinct eigenvalues, then its characteristic polynomial \eqref{eq:CharPoly} factors as
    $$P(x)=(x-\lambda)(x- \mu)$$
    with $\lambda \neq \mu$. For every point $p \in X$, the endomorphism $\phi_p$ has characteristic polynomial $P(x)$ and it is therefore diagonalizable with eigenvalues $\lambda$ and $\mu$. Hence 
    \begin{equation}\label{eq:isomfibers}
        \Theta_{X,p} = \mathrm{ker}(\phi_p- \lambda I_{\Theta_{X,p}}) \oplus \mathrm{ker}(\phi_p- \mu I_{\Theta_{X,p}}),
    \end{equation}
    and both eigenspaces have dimension one. Therefore, the morphisms 
    $$\phi - \lambda \cdot I_{\Theta_X}: \Theta_X \to \Theta_X, \quad \phi - \mu \cdot I_{\Theta_X}: \Theta_X \to \Theta_X$$
    have constant rank. Therefore $L_{\lambda}=\mathrm{ker}(\phi - \lambda \cdot I_{\Theta_X})$ and $L_{\mu}=\mathrm{ker}(\phi - \mu \cdot I_{\Theta_X})$ are line subbundles of $\Theta_X$. Moreover, the natural morphism $L_{\lambda} \oplus L_{\mu} \to \Theta_X$ is an isomorphism on every fiber by \eqref{eq:isomfibers}, concluding that
    $$\Theta_X \cong L_{\lambda} \oplus L_{\mu}.$$
    Conversely, suppose that $\Theta_X$ has the splitting $\Theta_X \cong L_1 \oplus L_2$. Let $p_1: \Theta_X \to L_1$ and $i_1: L_1 \to \Theta_X$, denote the projection and inclusion morphisms, respectively. Define $\pi=i_1 \circ p_1: \Theta_X \to \Theta_X$. For $v=(v_1, v_2) \in \Theta_X$ with $v_1 \in L_1$ and $v_2 \in L_2$, we have
    $$\pi(v_1, v_2)= i_1(p_1(v_1, v_2))=i_1(v_1)=(v_1, 0).$$
    Since $\pi(v_1,0)=(v_1,0)$ and $\pi(0,v_2)=0(0, v_2)$, then $\pi \in H^0(X, \mathrm{End}(\Theta_X))$ is a global endomorphism with eigenvalues $1$ and $0$. 
\end{proof}
\end{proposition}
The next result describes the space of global endomorphisms of the tangent bundle when the latter splits as a direct sum of line bundles. This will be used repeatedly in the examples and applications that follow.
\begin{theorem}\label{Thm:TXSplitEndo}
    Let $X$ be a smooth connected compact surface. If $\Theta_X \cong L_1 \oplus L_2$, then
    $$H^0(X, \mathrm{End}(\Theta_X)) \cong \CC \oplus H^0(X,L_1 \otimes L_2^{\vee}) \oplus H^0(X,L_2 \otimes L_1^{\vee}) \oplus \CC.$$
    In other words, every global endomorphism of $\Theta_X$ is represented by a matrix
    $$
    \left( \begin{array}{cc}
    a & \alpha \\
    \beta & b
    \end{array} \right),
    $$
    where $a,b \in \CC$, $\alpha \in H^0(X, L_1 \otimes L_2^{\vee})$ and $\beta \in H^0(X, L_2 \otimes L_1^{\vee})$.
    \begin{proof}
        Since $\Theta_X \cong L_1 \oplus L_2$,
        \begin{align*}
            H^0(X, \mathrm{End}(\Theta_X))&\cong H^0(X,\Theta_X \otimes \Theta_X^{\vee})\\
            &\cong H^0(X,(L_1 \oplus L_2) \otimes (L_1 \oplus L_2)^{\vee})\\
            &\cong H^0(X,L_1 \otimes L_1^{\vee}) \oplus H^0(X,L_1 \otimes L_2^{\vee}) \oplus H^0(X,L_2 \otimes L_1^{\vee}) \oplus H^0(X,L_2 \otimes L_2^{\vee})\\
            &\cong \CC \oplus H^0(X,L_1 \otimes L_2^{\vee}) \oplus H^0(X,L_2 \otimes L_1^{\vee}) \oplus \CC, 
        \end{align*}
        where the last equality is because 
        $$H^0(X,L_1 \otimes L_1^{\vee}) \cong H^0(X,L_2 \otimes L_2^{\vee}) \cong H^0(X, \mathcal{O}_X)\cong \CC,$$
        since $X$ is a connected compact surface.
    \end{proof}
\end{theorem}
We now illustrate Theorem \ref{Thm:TXSplitEndo} in two classical situations.
\begin{example}\label{Ex:Abelian}
    If $X$ is an abelian surface, then $\Theta_X \cong \mathcal{O}_X \oplus \mathcal{O}_X$ and $H^0(X, \mathrm{End}(\Theta_X)) \cong \CC^4$. In other words, every global endomorphism of $\Theta_X$ is represented by
    $$
    \left( \begin{array}{cc}
    a & b \\
    c & d
    \end{array} \right), \quad a,b,c,d \in \CC.
    $$
\end{example}
\begin{example}\label{Ex:P1P1}
    Consider $X= \PP^1 \times \PP^1$. The Euler sequence on $\PP^1$ given by
    $$0 \to \mathcal{O}_{\PP^1} \to \mathcal{O}_{\PP^1}(1)^{\oplus 2} \to \Theta_{\PP^1} \to 0,$$
    implies that $c_1(\Theta_{\PP^1})=2$ and hence $\Theta_{\PP^1} \cong \mathcal{O}_{\PP^1}(2)$. Therefore
    $$\Theta_X \cong p_1^{\ast}\Theta_{\PP^1} \oplus p_2^{\ast}\Theta_{\PP^1} \cong p_1^{\ast} \mathcal{O}_{\PP^1}(2) \oplus p_2^{\ast} \mathcal{O}_{\PP^1}(2),$$
    see \cite[Ex 8.3 a), p.187]{Hartshorne} for a reference. By Theorem \ref{Thm:TXSplitEndo} we have that
    $$H^0(X, \mathrm{End}(\Theta_X)) \cong \CC \oplus H^0(X, p_1^{\ast}\mathcal{O}_{\PP^1}(2) \otimes p_2^{\ast}\mathcal{O}_{\PP^1}(-2)) \oplus H^0(X, p_1^{\ast}\mathcal{O}_{\PP^1}(-2) \otimes p_2^{\ast}\mathcal{O}_{\PP^1}(2)) \oplus \CC.$$
    By Kunneth formula we have that 
    $$H^0(X, p_1^{\ast}\mathcal{O}_{\PP^1}(2) \otimes p_2^{\ast}\mathcal{O}_{\PP^1}(-2)) \cong H^0(\PP^1,\mathcal{O}_{\PP^1}(2)) \otimes H^0(\PP^1,\mathcal{O}_{\PP^1}(-2)) \cong 0,$$
    and similarly $H^0(X, p_1^{\ast}\mathcal{O}_{\PP^1}(-2) \otimes p_2^{\ast}\mathcal{O}_{\PP^1}(2)) \cong 0$. Hence $H^0(X, \mathrm{End}(\Theta_X)) \cong \CC^2  $. In other words, every global endomorphism of $\Theta_X$ is represented by
    \begin{equation}\label{eq:MatrixQ}
        \left( \begin{array}{cc}
    a & 0 \\
    0 & b
    \end{array} \right), \quad a,b \in \CC.
    \end{equation}
    This computation provides an alternative proof of Theorem 4.1 of \cite{Hirzebruch} in the case $\delta=0$.
\end{example}

\begin{remark}\label{rmk:QQuadricP}
Every smooth irreducible quadric in $\PP^3$ is isomorphic to 
$Q=V(xy-zw) \subset \PP^3$, and $Q \cong \PP^1 \times \PP^1$, by 
\cite[Ex 2.15 a), p.13]{Hartshorne}. Example \ref{Ex:P1P1} shows that every global endomorphism of $\Theta_Q$ is represented by a matrix of the form \eqref{eq:MatrixQ}.
By Theorem \ref{Th:RigiditySurfaces}, for $d > 1$
there do exist different foliations on $Q$ with isolated singularities 
and tangent sheaf $i^{\ast}\mathcal{O}_{\PP^3}(1-d)$ with the same singular scheme.
\end{remark}

\begin{remark}
A straightforward computation using formula 
\eqref{CanonicalBundle} shows that the only smooth projective complete intersection surfaces with ample anticanonical bundle (namely, the complete intersection \emph{Del Pezzo} surfaces) are $\PP^1\times\PP^1$, the smooth cubic surfaces in $\PP^3$, and the smooth complete intersections of type $(2,2)$ in $\PP^4$.
    
The first case was discussed in Example \ref{Ex:P1P1}
and Remark \ref{rmk:QQuadricP}. The remaining two cases are covered by \cite[Théorème]{Fah} and were addressed in Corollaries \ref{Cor:EndoHyperP3} and \ref{Cor:Endo22}. Therefore, we have just given
a precise description of the spaces $H^0(X, \mathrm{End}(\Theta_X))$ 
for all smooth projective complete intersection Del Pezzo surfaces $ X $.
\end{remark}

\begin{proposition}\label{Prop:NoSplit}
Let $X \subset \PP^n$ be a smooth projective complete intersection surface with $\mathrm{Pic}(X) \cong \ZZ$. Then the tangent bundle $\Theta_X$ does not split as a direct sum of two line bundles.
\begin{proof}
Suppose that $\Theta_X$ splits and $\mathrm{Pic}(X) \cong \ZZ$, then 
$$\Theta_X \cong \mathcal{O}_X(a) \oplus \mathcal{O}_X(b),$$
for some $a,b \in \ZZ$, then 
\begin{equation}\label{eq:Chern2}
c_1(\Theta_X)= (a+b)H \quad  c_2(\Theta_X)= abH^2.
\end{equation}
Comparing \eqref{eq:Chern2} with the Chern classes of $\Theta_X$ computed 
during
the proof of Theorem \ref{eq:NumberofSing}, we obtain
\begin{equation}\label{eq:Chern3}
\begin{split}
    a+b&= n+1-\sum_{i=1}^c d_i, \quad and\\
    ab&= \binom{n+1}{2}-(n+1)\sum_{i=1}^c d_i +\sum_{i=1}^c d_i^2+\sum_{i<j} d_id_j.
\end{split}
\end{equation}
Let
$$S:=\sum_{i=1}^c d_i \quad T:=\sum_{i=1}^c d_i^2.$$
Then $a$ and $b$ are roots of the polynomial
$$R(x)=(x-a)(x-b)=x^2-(n+1-S)x+\left( \binom{n+1}{2}-(n+1)S+T+\sum_{i<j}d_i d_j \right).$$
The discriminant of $R$ is given by
\begin{equation*}
    \begin{split}
        \Delta_R =& (n+1-S)^2-4\left(\binom{n+1}{2}-(n+1)S+T+\sum_{i<j}d_i d_j \right)\\
        =& -n^2+1+2(n+1)S-S^2-2T,
    \end{split}
\end{equation*}
by the identity $S^2=T+2\sum_{i<j}d_i d_j$.
Since $c=n-2$ and $d_i \geq 2$ for each $i=1, \ldots, c$, we have $S \geq 2(n-2)$. Moreover, $T \geq \frac{S^2}{c}$ by the
Cauchy-Schwarz inequality and hence 
$-2T \leq \frac{-2S^2}{n-2}$. Therefore
\begin{equation*}
    \begin{split}
        \Delta_R =& -n^2+1+2(n+1)S-S^2-2T\\
        \leq& -n^2+1+2(n+1)S-S^2-\frac{2S^2}{n-2}\\
        =&-\frac{n}{n-2}S^2+2(n+1)S+1-n^2.
    \end{split}
\end{equation*}
Let us fix $n \geq 3$, then 
$$U(S):=-\frac{n}{n-2}S^2+2(n+1)S+1-n^2$$
is a quadratic polynomial in $S$ with negative principal coefficient, whose maximum value is
$$-2\frac{n+1}{n}<0.$$
Hence
$\Delta_R \leq U(S)<0$ for every $n \geq 3$. It follows that $R$ has no real roots, contradicting the existence of the integers $a$ and $b$. Therefore $\Theta_X$ cannot split.
\end{proof}
\end{proposition}

By the Noether-Lefschetz theorem, a \emph{very general} smooth projective complete intersection surface has Picard rank one, except for quadrics in $\PP^3$, cubic surfaces in $\PP^3$, and complete intersections of two quadrics in $\PP^4$ (see \cite[Theorem 0.1]{VeryGeneral} for a reference), which were discussed above. Therefore, Proposition \ref{Prop:NoSplit} shows that, for a very general smooth projective complete intersection surface, the tangent bundle does not split as a direct sum of two line bundles.

The following proposition characterizes global endomorphisms of $\Theta_X$ having a unique eigenvalue.
\begin{proposition}\label{prop:Nilpot}
    Let $X$ be a smooth connected compact surface. $\phi \in H^0(X, \mathrm{End}(\Theta_X))$ has a unique eigenvalue $\lambda$, if and only if $\phi= \lambda I_{\Theta_X} +N$, where $N \in H^0(X, \mathrm{End}(\Theta_X))$ satisfies $N^2=0$.
\begin{proof}
    Assume that $\phi$ has only one eigenvalue, then its characteristic polynomial \eqref{eq:CharPoly} is
    $$P(x)=(x-\lambda)^2,$$
    for some $\lambda \in \CC$. For every point $p \in X$, the endomorphism $\phi_p$ has characteristic polynomial $P(x)$ and by Cayley-Hamilton theorem, $\phi_p$ satisfies its characteristic polynomial, that is
    $$P(\phi_p)=(\phi_p-\lambda I_{\Theta_{X,p}})^2=0.$$
    Setting $N=\phi-\lambda I_{\Theta_X}$, we obtain $(N^2)_p=(N_p)^2=0$ for every $p \in X$. Therefore $N^2=0$ and $\phi=\lambda I_{\Theta_X} +N$.

    Conversely, suppose that $\phi=\lambda I_{\Theta_X}+N$ with $N^2=0$. Since $(N_p)^2=0$ for every $p \in X$, then $ N_p \sim
\left(\begin{smallmatrix}
0 & 1\\
0 & 0
\end{smallmatrix}\right)
$ or $N_p=0$. Hence 
$$\phi_p= \lambda I_{\Theta_{X,p}} +N_p \sim \begin{pmatrix}
\lambda & 1\\
0 & \lambda \\
\end{pmatrix}
$$
or $ \phi_p \sim
\left(\begin{smallmatrix}
\lambda & 0\\
0 & \lambda
\end{smallmatrix}\right)
$. Since similar matrices have the same characteristic polynomial
$$P_p(x)=\mathrm{det}\begin{pmatrix}
x-\lambda & \ast\\
0 & x-\lambda \\
\end{pmatrix}=(x- \lambda)^2.$$
Therefore $\phi$ has only one eigenvalue.
\end{proof}
\end{proposition}

\begin{remark}
    If $N=0$, then $\phi=\lambda I_{\Theta_X}$ is a scalar multiple of the identity.
\end{remark}

\begin{corollary}\label{Cor:LambdaZero}
    Let $X$ be a smooth connected compact surface and let $\phi \in H^0(X, \mathrm{End}(\Theta_X))$. Then $\phi$ has $0$ as its unique eigenvalue if and only if $\phi^2=0$.
    \begin{proof}
        If $0$ is the unique eigenvalue, Proposition \ref{prop:Nilpot} gives $\phi=N$ and $\phi^2=N^2=0$. Conversely, assume that $\phi^2=0$. For every $p \in X$, the fiber map $\phi_p$ is a nilpotent map whose only eigenvalue is $0$, hence $0$ is the unique eigenvalue of $\phi$.
    \end{proof}
\end{corollary}

\begin{corollary}\label{Cor:Autom}
     Let $X$ be a smooth connected compact surface and let $\phi \in H^0(X, \mathrm{End}(\Theta_X))$. Assume that $\phi$ has a unique eigenvalue $\lambda$. Then $\lambda \neq 0$ if and only if $\phi$ is an automorphism of $\Theta_X$. In this case,
     $$\phi^{-1}=\frac{1}{\lambda} I_{\Theta_X}-\frac{1}{\lambda^2}N,$$
     where $\phi=\lambda I_{\Theta_X}+N$.
     \begin{proof}
         By Proposition \ref{prop:Nilpot}, 
         $$\phi=\lambda I_{\Theta_X}+N=\lambda \left(I_{\Theta_X}+\frac{1}{\lambda} N \right)$$ 
         with $N^2=0$. Since 
         $$\left(I_{\Theta_X}+\frac{1}{\lambda} N\right)\left(I_{\Theta_X}-\frac{1}{\lambda} N\right)=I_{\Theta_X}-\frac{1}{\lambda^2}N^2=I_{\Theta_X}.$$
         Therefore,
         $$\left(\frac{1}{\lambda} I_{\Theta_X}-\frac{1}{\lambda^2}N\right) \circ \phi= \phi \circ \left(\frac{1}{\lambda} I_{\Theta_X}-\frac{1}{\lambda^2}N\right)= I_{\Theta_X},$$
         showing that
         $$\phi^{-1}=\frac{1}{\lambda} I_{\Theta_X}-\frac{1}{\lambda^2}N.$$
         Hence $\phi$ is an automorphism of $\Theta_X$.

         Conversely, if $\phi$ is an automorphism, then each fiber map $\phi_p$ is invertible. Hence $0$ cannot be an eigenvalue of any $\phi_p$. Since $\phi$ has a unique eigenvalue, it must be nonzero.
     \end{proof}
\end{corollary}

\section*{Acknowledgments}
The second author gratefully acknowledges the support of the Centro de Investigación en Matemáticas, A.C., through a Research Stay Grant during the preparation of this manuscript.

\end{document}